\documentclass[12pt]{article}

 \usepackage[T1]{fontenc}              
 \usepackage[width=16cm, height=22cm]{geometry}   

 \usepackage{mathtools} 
 \usepackage{amssymb}        
 \usepackage[amsmath,thmmarks,hyperref]{ntheorem} 

 \usepackage{icomma}                   
 \usepackage{enumerate}          
 \usepackage{color}                 
 \usepackage{setspace}                 
 \usepackage{needspace}               
 \usepackage{url}                    
 \usepackage[mathscr]{euscript}              
 \usepackage{esint}    
 
 \usepackage[numbers,square]{natbib}

\usepackage{tikz}
\usetikzlibrary{arrows.meta}
\usetikzlibrary{cd}  

\usepackage{hyperref} 

\theoremstyle{nonumberplain}  
\theoremheaderfont{\itshape}  
\theorembodyfont{\normalfont}  
\theoremseparator{.}  
\theoremsymbol{$\Box$}
\newtheorem{proof}{Proof} 
  
\theoremstyle{plain}  
\theoremheaderfont{\normalfont\bfseries}  
\theorembodyfont{\itshape}  
\theoremsymbol{~}
\theoremseparator{.}  
\newtheorem{proposition}{Proposition}[section]  
  
\newtheorem{corollary}[proposition]{Corollary}  
\newtheorem{lemma}[proposition]{Lemma}  
\newtheorem{theorem}[proposition]{Theorem}

\theorembodyfont{\normalfont}  
 
\newtheorem{remark}[proposition]{Remark}
\newtheorem{remarks}[proposition]{Remarks}
\newtheorem{example}[proposition]{Example}  
\newtheorem{examples}[proposition]{Examples}

\newtheorem{definition}[proposition]{Definition} 
\newtheorem{notation}[proposition]{Notation} 

\newcommand{\cl}{\mathrm{cl}}

\newcommand{\bbR}{\mathbb{R}}
\newcommand{\Det}{\mathrm{Det}}
\newcommand{\Hom}{\mathrm{Hom}}
\newcommand{\bbC}{\mathbb{C}}
\newcommand{\bbZ}{\mathbb{Z}}

\newcommand{\Res}{R}

\DeclareMathOperator{\CAR}{CAR}
\DeclareMathOperator{\Cl}{Cl}

\DeclareMathOperator{\Id}{Id}
\DeclareMathOperator{\image}{im}

\DeclareMathOperator{\Top}{top}

\begin{document}
\title{The Chiral Anomaly of the Free Fermion in Functorial Field Theory}
\author{Matthias Ludewig\footnote{Universit\"at Greifswald, E-Mail: matthias.ludewig@uni-greifswald.de} ~and Saskia Roos}

\maketitle

\begin{abstract}
When trying to cast the free fermion in the framework of functorial field theory, its {\em chiral anomaly} manifests in the fact that it assigns the determinant of the Dirac operator to a top-dimensional closed spin manifold, which is not a number as expected, but an element of a complex line. In functorial field theory language, this means that the theory is {\em twisted}, which gives rise to an {\em anomaly theory}. In this paper, we give a detailed construction of this anomaly theory, as a functor that sends manifolds to infinite-dimensional Clifford algebras and bordisms to bimodules. 
\end{abstract}

\tableofcontents

\section{Introduction}

Functorial field theory is a mathematical formalism (first introduced by Atiyah and Segal \cite{AtiyahTFT, Segal}) which is designed to capture aspects of physical quantum field theory. Concisely, a functorial field theory is a  monoidal functor from a suitable bordism category to a linear category, e.g.\ the category of vector spaces. Spelled out, such a field theory of dimension $d$ assigns to a $(d-1)$-dimensional manifold $Y$ a vector space $F(Y)$ (the ``space of states''), and to a bordism $X$ between manifolds $Y_0$, $Y_1$ a linear map $F(X): F(Y_1) \rightarrow F(Y_0)$ (which describes the time evolution of these states). {\em Functoriality} then means that these maps behave as expected under gluing of bordisms. {\em Monoidality} of the functor means that $F$ sends disjoint unions to tensor products. The latter implies in particular that $F$ assigns to a closed manifold $X$ (seen as a bordism from the empty set to itself) a number, hence $F$ determines a function on the moduli space of closed $d$-manifolds, the {\em partition function}, and it is a general principle that this function determines key aspects of the theory.

In physics, a field theory is often described in terms of an {\em action functional} $S$ on a {\em space of fields} $\mathscr{F}(X)$, and it is a somewhat fundamental question if and how such a field theory (in the physicist's sense) can be described mathematically by a functorial field theory; more concretely, given for each manifold $X$ a space of fields $\mathscr{F}(X)$ together with an action functional $S$ on this space (both of these should be ``local'' and ``physical'' in a suitable sense), one asks for a canonical way to construct a symmetric monoidal functor $F$ as described above from these data. At least formally, the partition function of $F$ is easily described: A physicist's paradigm states that the partition function of this functorial field theory should be given by the {\em path integral} \cite[ \S 1.2]{MuellerSzabo},
\begin{equation} \label{PathIntegral}
  F(X) = \int_{\mathscr{F}(X)} e^{-S(\varphi)} \mathcal{D} \varphi.
\end{equation}
This already illustrates that the task of constructing $F$ is far from trivial: The spaces of fields $\mathscr{F}(X)$ are typically infinite-dimensional, and integration over them is not defined.

\medskip

Work-arounds are possible, however, in the case where $\mathscr{F}(X)$ is a linear space and the action functional $S$ is quadratic. Maybe the simplest, already non-trivial, example of this type is the {\em free scalar field}, where $\mathscr{F}(X) = C^\infty(X)$ and the action functional is 
\begin{equation*}
  S(\varphi) = \int_X \bigl(|d\varphi|^2 + m^2 \varphi^2\bigr) = \int_X \varphi\cdot (\Delta + m^2)\varphi, \qquad m >0,
\end{equation*}
where for the second equality to hold in general, we need $X$ to be closed.
In this case, the path integral \eqref{PathIntegral} has the form of a Gaussian integral, so we just stipulate
\begin{equation*}
  F(X) := \det\nolimits_\zeta(\Delta + m^2)^{-1/2},
\end{equation*}
by analogy with the finite-dimensional case; here the right hand side is the determinant of the unbounded operator $\Delta + m^2$, defined using zeta-regularization. The task of defining a functorial field theory with this partition function is an intriguing story; it has been completed by Kandel \cite{Kandel}. Let us remark however,  that while this field theory is very interesting analytically, it has no topological content, since it can be continuously deformed to the trivial field theory.

Things are different for the {\em free fermionic field}. Here $X$ is an even-dimensional spin manifold and $\mathscr{F}(X) = C^\infty(X, \Sigma_X)$ is the space of smooth spinors of $X$. The action functional is
\begin{equation*}
  S(\Phi) = \int_X \langle \Phi, D_X\Phi\rangle,
\end{equation*}
where $D_X$ is the Dirac operator. While this action functional is still quadratic, the corresponding operator $D_X$ is no longer positive, which entails that the determinant $\det\nolimits_\zeta(D_X)^{-1/2}$ is not a number in a natural way, but instead an element of the (conjugate of the) {\em determinant line} 
\begin{equation*}
\mathrm{Det}_X  = \overline{\Lambda^{\mathrm{top}}\mathcal{H}_X^+} \otimes \Lambda^{\mathrm{top}} \mathcal{H}_X^-,
\end{equation*}
where $\mathcal{H}_X^{\pm}$ are the harmonic spinors on $X$ (of positive/negative chirality). From a supergeometry point of view, this is due to the fact that the Berezin integral does not canonically produce a number without further choices.
This is one manifestation of the {\em chiral anomaly} of the free fermion, and it means that the corresponding functorial field theory must be {\em twisted}.

\medskip

Roughly, a $d$-dimensional twisted field theory is a theory that, instead of a vector space, assigns an $A_Y$-module to a $(d-1)$-dimensional manifold $Y$, where the twisting manifests in the fact that the algebra $A_Y$ may depend on the manifold $Y$ as well. 
To explain this rigorously, we have to start with the notion of a {\em twist}. For the free fermion in $d$ dimensions, this is a functor
\begin{equation*}
  T: \mathrm{Bord}_{\langle d-1, d\rangle}^{\mathrm{Spin}} \longrightarrow \mathrm{sAlg}
\end{equation*}
from the $d$-dimensional spin bordism category to the bicategory $\mathrm{sAlg}$ of $\bbZ_2$-graded algebras, $\bbZ_2$-graded bimodules and grading-preserving intertwiners. In particular, $T$ assigns an algebra $T(Y)$ to a $(d-1)$-dimensional spin manifold $Y$ and a $T(Y_0)$-$T(Y_1)$-bimodule $T(X)$ to a $d$-dimensional spin bordism $X$ between $Y_1$ and $Y_0$.
A $d$-dimensional $T$-twisted field theory is then a natural transformation
\begin{equation*}
\begin{tikzcd}[column sep = {1.7cm,between origins},row sep={0.5cm,between origins}]
 & \,\ar[dd, Rightarrow, "F"] & \\
\mathrm{Bord}_{\langle d-1, d\rangle}^{\mathrm{Spin}}~~ \ar[rr, bend left=30, "\mathbf{1}"] \ar[rr, bend right=30, "T"'] & & \mathrm{sAlg}, \\
 & \, &
\end{tikzcd}
\end{equation*}
where $\mathbf{1}$ is the trivial twist that assigns $\bbC$ (considered as an algebra, respectively a $\bbC$-$\bbC$-bimodule) to all manifolds and bordisms. In the case that $T$ is also the trivial twist, this definition reduces to the definition of an untwisted field theory discussed above. In general however, for a $(d-1)$-dimensional spin manifold $Y$, $F(Y)$ is not a vector space but a $T(Y)$-module. In the case where $X$ is a closed $d$-manifold, $T(X)$ will be a $\bbC$-$\bbC$-bimodule, i.e.\ a vector space, and instead of a number, $F(X)$ will be an element of this vector space. 

\medskip

This paper is dedicated to the construction of the twist $T$ of the free fermion. Since the free fermion will assign the determinant of the Dirac operator to a closed $d$-dimensional spin manifold $X$, which is an element of the conjugate of the determinant line, our discussion above shows that we must have 
\begin{equation*}
T(X) = \overline{\mathrm{Det}}_X
\end{equation*}
 for closed manifolds $X$. In total, $T$ will be defined as follows. First, to a $(d-1)$-dimensional spin manifold $Y$, $T$ assigns the Clifford algebra $\Cl(W_Y)$ on the space $W_Y$ of smooth spinors on $Y$, which is isomorphic to a suitable algebra of canonical anti-commutation relations (a CAR algebra). To obtain a $\Cl(W_{Y_0})$-$\Cl(W_{Y_1})$-bimodule for a spin bordism $X$ between $Y_1$ and $Y_0$, one notices that the space $L_X$ of boundary values of harmonic spinors on $X$ is a {\em Lagrangian} in the direct sum of $W_{Y_0}$ and $W_{Y_1}$, so that the exterior algebra $\Lambda L_X$ is such a bimodule in a natural way (Thm.~\ref{TheoremLagrangian}). 

The main issue is now functoriality, i.e.\ the behavior of the bimodules $T(X)$ under gluing of bordisms. Here we define isomorphisms
\begin{equation*}
 \tau: T(X) \longrightarrow T(X_0) \otimes_{T(Z)} T(X_1)
\end{equation*}
when $X$ is obtained by gluing two bordisms $X_0$, $X_1$ along a common boundary $Z$. Since the target $\mathrm{sAlg}$ of $T$ is a bicategory, we cannot expect to have equality here; instead, the twist $T$ includes these isomorphisms $\tau$ as additional data. Moreover, we show that these isomorphisms are  {\em coherent} in the sense that whenever a bordism is decomposed into three pieces, the two different ways of composing the corresponding gluing isomorphisms coincide. This is proved in Thm.~\ref{ThmCoherence}.

\medskip

One of the main observations of this paper is that the chiral anomaly appears due to a {\em purely algebraic} reason, which is that ``second quantization'' of fermions is functorial only up to a certain error. The Gluing Theorem~\ref{ThmGluing}, an abstract result on composition of Clifford modules coming from Lagrangian subspaces, provides a clear understanding of this phenomenon. 
Together with a corresponding coherence result, Thm.~\ref{ThmNumber}, it is the backbone of our construction of the functor $T$. 
The first of these results is an extension of the Gluing Lemma~2.2.8 of Stolz and Teichner \cite{StolzTeichnerElliptic}; the latter seems to be entirely new. 

To begin with, we notice that it is not even clear that the composition of two Lagrangians is again a Lagrangian; our Thm.~\ref{ThmComposition} gives sufficient conditions for this to be the case, which leads to a good category of Lagrangian relations in the infinite-dimensional setup, c.f.\ Remark~\ref{RemarkCategory}. Conveniently, these sufficient conditions turn out to also be necessary in order to have the Gluing Theorem~\ref{ThmGluing}.

We remark that our constructions are purely algebraic: Our spaces $W$ will be complex pre-Hilbert spaces with a real structure, which are not required to be complete with respect to the Hermitian form; the Clifford algebra is a quotient of the algebraic tensor algebra of $W$, and all tensor products are algebraic.

\medskip

We expect that the Gluing Theorem~\ref{ThmGluing} can be extended to a functional analytic setup, where the Clifford algebras are realized as von Neumann algebras, the bimodules are Hilbert space bimodules over these and the algebraic tensor product is replaced by Connes' fusion product. This is a question of ongoing research.

The anomaly theory $T$ constructed in this paper is part of a much larger story: Conjecturally, it can be extended above and below, to an {\em extended} functorial field theory in a higher categorical framework. Its next higher level, for example, is the theory constructed by Dai-Freed in \cite{FreedDai}. It is a fascinating observation, sketched in \S\ref{SectionOutlook} below, that all of the index theory of the Dirac operator is comprised in the anomaly theory of the free fermion.

\medskip

This paper is structured as follows. In \S\ref{Section Clifford Algebras}, we discuss the algebraic preliminaries,  and prove the abstract results on the composition of bimodules over Clifford algebras needed later. In \S\ref{SectionSpinGeometry}, we first introduce the geometric setup we will be working in, followed by a discussion of the analysis of boundary value problems of the Dirac operator needed in this paper. The construction of the functor $T$ is then carried out in \S\ref{SectionTwist}. Finally, in \S\ref{SectionOutlook}, we give an outlook on further aspects of the story that are not discussed in detail in this paper.

\medskip

The paper \cite{MickelssonScott}, where Mickelsson and Scott construct a version of the free fermion, contains many similar ideas compared to this paper. However, they attempt to construct it as an untwisted usual functorial field theory, an endeavor which only yields a {\em projective} functor owing to the presence of the anomaly.

The article \cite{MuellerSzabo} also considers the anomaly of the free fermion in functorial field theory setting, but a different part compared to our paper: In the language of \S\ref{SectionFreedDaiTheory}-\S\ref{SectionSecondExtension}, their theory is essentially the $\langle d, d+1, d+2\rangle$-dimensional part of the free fermion, while we focus on the $\langle d-1, d\rangle$-dimensional part.

There is a rich literature on the free fermion two-dimensional, conformal case, starting with Segal \cite{Segal}. In particular, the papers \cite{Kriz} and \cite{Tener} have some similarities to ours, but focus on modularity formulas, respectively traciality properties. In this paper, we focus on the general picture, ignoring special features of the two-dimensional case.

\paragraph{Acknowledgements.} It is our pleasure to thank  C.\ B\"ar, P.\ Kristel, A.\ Hermann, E.\ Rabinovich, A.\ Stoffel, S.\ Stolz, P.\ Teichner and K.\ Waldorf for helpful discussions. 
We are further indebted to the Max-Planck-Institute in Bonn, where part of this research was conducted, as well as the University of Potsdam and the University of Adelaide. The first-named author was supported by the Max-Planck-Foundation and the ARC Discovery Project grant FL170100020 under Chief Investigator and Australian Laureate Fellow Mathai Varghese. The second named author held a Hausdorff Scholarship provided by the Hausdorff Center of Mathematics in Bonn and was also supported by the Special Priority Programme SPP2026, ``Geometry at Infinity'' of the DFG.

\paragraph{Acknowledgment for v4.}

I thank Raphael Schmidpeter for pointing out that the published proof of Thm.~\ref{ThmNumber} was incorrect, and for helping to obtain a correct version, which is now contained in arXiv version 4.

\section{Clifford Algebras and their Modules}\label{Section Clifford Algebras}

This section contains the algebraic part of the paper. We start by discussing Clifford algebras in the setting of complex vector spaces with a real structure, and afterwards Lagrangians in these spaces. Finally, these two notions are brought together when we explain how Lagrangians give rise to bimodules over Clifford algebras. Throughout, we take care to not require completeness of our vector spaces, since subsequently, we aim to apply the results of this section to the space of smooth spinors on a compact manifold, endowed with the $L^2$ inner product.

\subsection{Real Structures and Clifford Algebras} \label{SectionRealStructures}

Let $W$ be a (possibly infinite-dimensional and not necessarily complete) Hermitian vector space, and let $\overline{W}$ be its complex conjugate, which has the same underlying vector space, but the complex structure replaced by its negative. The identity map $W \rightarrow \overline{W}$ is an $\bbR$-linear, $\bbC$-{\em antilinear} vector space isomorphism between the two. 
We will denote the element of $\overline{W}$ corresponding to $v \in W$ by ${v}^*$, and similarly, if $\xi \in \overline{W}$, we denote by ${\xi}^*$ the same element, considered as an element of $W$. 
The vector space $\overline{W}$ carries an induced Hermitian form, given by
\begin{equation*}
  \langle \xi_1, \xi_2 \rangle_{\overline{W}} \coloneqq \langle {\xi}_2^*, {\xi}_1^* \rangle_W = \overline{\langle {\xi}_1^*, {\xi}_2^* \rangle}_W, \ \ \ \ \  \xi_1, \xi_2 \in \overline{W}.
\end{equation*}
We remark that our convention is that Hermitian forms are $\bbC$-antilinear in the {\em first} entry.

\begin{definition}[Real structure]\label{DefinitionRealStructure}
 A {\em real structure} on a Hermitian vector space $W$ is an anti-unitary involution 
 \begin{equation*}
   W \longrightarrow W, \qquad w \longmapsto \overline{w}.
 \end{equation*}
The {\em opposite} $-W$ of a space $W$ with a real structure consists of the same underlying Hermitian vector space together with the real structure $w \mapsto - \overline{w}$. If $f: V \rightarrow W$ is a map between real vector spaces, then its {\em conjugate} $\overline{f}:V \rightarrow W$ is defined by
\begin{equation*}
  \overline{f}(v) = \overline{f(\overline{v})}.
\end{equation*}
\end{definition}

\begin{notation}[The bilinear form] \label{NotationBracket}
On a Hermitian vector space $W$, a real structure induces a complex-bilinear form $b$, defined by
\begin{equation}\label{Eq.Bracket}
  b(v, w) \coloneqq \langle \overline{v}, w \rangle, \ \ \ \ \ v, w \in W.
\end{equation}
\end{notation}

This allows us to form the corresponding Clifford algebra.

\begin{definition}[Clifford algebra] 
Given a Hermitian vector space $W$ with a real structure and associated bilinear form $b(\cdot, \cdot)$, we denote by $\Cl(W)$ the Clifford algebra associated to the bilinear form defined above; in other words, the algebra generated by the elements of $W$, with the relation
  \begin{equation} \label{CliffordRelationBracket}
     v \cdot w + w \cdot v = b( v, w)
  \end{equation}
for $v, w \in W$. Declaring elements of $W\subset \Cl(W)$ to be odd induces a $\bbZ_2$-grading on the Clifford algebra.
\end{definition}

Recall that the {\em opposite} of an algebra $A$ is the algebra $A^{\mathrm{op}}$, which has the same underlying vector space, but the multiplication reversed. If $A$ is a superalgebra, the opposite is defined with an additional sign; more precisely, $A^{\mathrm{op}}$ is the superalgebra that has the same underlying vector space, but multiplication  $\bullet$ defined by
\begin{equation} \label{SignInOppMult}
  a \bullet b = (-1)^{|a||b|} b a
\end{equation}
for homogeneous elements $a, b \in A^{\mathrm{op}}$. When applying this to our Clifford algebras, we have canonically
\begin{equation} \label{OppositeAlgebraEq}
 \Cl(W)^{\mathrm{op}}  \cong \Cl(-W),
\end{equation}
in other words, the opposite of the Clifford algebra associated to a Hermitian space $W$ with a real structure  is the Clifford algebra associated to the opposite  $-W$, see Def.~\ref{DefinitionRealStructure}.

\begin{remark}[The CAR algebra] \label{CARalgebra}
Given a Hermitian vector space $V$, the space $W \coloneqq V \oplus \overline{V}$ carries a canonical real structure, given by the {\em flip map}
\begin{equation}\label{Eq.Flipmap}
  \overline{(v, \xi)} \coloneqq ({\xi}^*, {v}^*), \ \ \ \  v \in V, \ \xi \in \overline{V}.
\end{equation}
The associated bilinear form then takes the form
\begin{equation*}
   b\bigl( (v, \xi), (w, \eta)\bigr) = \bigl\langle \overline{(v, \xi)}, (w, \eta)\bigr\rangle_W = \langle  {\xi}^*, w\rangle_V + \langle {\eta}^*, v \rangle_V,
\end{equation*}
for $v, w \in V$, $\xi, \eta \in \overline{V}$. In this case, the Clifford algebra $\Cl(W)$  is canonically isomorphic to the algebra $\CAR(V)$, which is the algebra on the symbols $a(v)$, $a^*(v)$, for $v \in V$, subject to the {\em canonical anticommutation relations}
  \begin{equation*}
  \begin{aligned}
    a(v)a(w) + a(w) a(v) &= 0,\\
     a^*(v)a^*(w) + a^*(w) a^*(v) &= 0, \\
    a^*(v) a(w) + a^*(w)a(v) &= \langle v, w\rangle_V,
    \end{aligned}
  \end{equation*}
  for $v, w \in V$. Notice here that the assignment $v \mapsto a(v)$ is $\bbC$-linear, while the map $v \mapsto a^*(v)$ is $\bbC$-antilinear. An isomorphism to the algebra $\Cl(W)$ is provided by sending
  \begin{equation*}
     a(v) + a^*(w) \longmapsto (v, \overline{w})
  \end{equation*}
  for $v, w \in V$ and extending by multiplicativity.
\end{remark}

\subsection{Lagrangians and their Composition} \label{SectionLagrangians}

We start with the following definition.

\begin{definition}[Lagrangian]
Let $W$ be a complex vector space with a real structure. A complex subspace $L \subset W$ is called {\em Lagrangian} if $\overline{L} = L^\perp$.
\end{definition}

Above, the orthogonal complement is taken with respect to the Hermitian structure, not the bilinear form $b$ defined in \eqref{Eq.Bracket}.
This implies in particular that $b(v, w) = 0$ for all $v, w \in L$, i.e.\ that $L$ is a {\em totally isotropic} subspace with respect to the bilinear form. In fact, a Lagrangian $L$ is always maximal among totally isotropic subspaces, in the sense that if $L \subset L^\prime$ for a totally isotropic subspace $L^\prime$, then $L = L^\prime$. Notice, a Lagrangian is always a closed subspace, as it is an orthogonal complement.

We will mainly deal with the situation where we have two complex vector spaces $W_0$, $W_1$ with real structure  and a Lagrangian $L_{01} \subset W_0 \oplus - W_1$. This allows to speak of the {\em composition} of two Lagrangians.

\begin{definition}[Composition]
Let $W_0, W_1, W_2$ be complex vector spaces with real structures and let $L_{01} \subset W_0 \oplus - W_1$ and $L_{12} \subset W_1 \oplus - W_2$ be Lagrangians. The {\em composition} of the two Lagrangians is the subspace $L_{02} = L_{01} \circ L_{12} \subset W_0 \oplus - W_2$ defined by
\begin{equation*}
L_{02}\coloneqq \bigl\{ (w_0, w_2) \mid \exists w_1 \in W_1 :(w_0, w_1) \in L_{01}, (w_1, w_2) \in L_{12}\bigr\},
\end{equation*}
in other words, their composition as linear relation.
\end{definition}

While it is easy to see that $L_{02}$ is always an isotropic subspace, i.e.\ the bilinear form vanishes identically on $L_{02}$, it is not clear that $L_{02}$ is also maximal. The following example illustrates that this is not always the case.

\begin{example} \label{ExampleNonComposable}
Consider the Hilbert space $\ell^2(\bbZ)$, with the real structure given by $\overline{e}_n = e_{-n}$, where $e_n$, $n \in \bbZ$, is the $n$-th standard basis vector. For $\alpha \in \bbR$, define the unbounded operator $Q_\alpha$ with domain
\begin{equation*}
  \mathrm{dom}(Q_\alpha) = \left\{ (a_n)_{n \in \bbZ} ~ \left| ~ \sum_{n \in \bbZ} (1+ e^{2\alpha n}) |a_n|^2 < \infty\right. \right\} \subseteq \ell^2(\bbZ), 
\end{equation*}
given by $Q_\alpha e_n = e^{\alpha n} e_n$. It is easy to check that $Q_\alpha$ is densely defined and closed, and therefore, the property $Q_\alpha^* = \overline{Q}_\alpha^{-1}$ shows that 
\begin{equation*}
L_\alpha := \mathrm{graph}(Q_\alpha)
\end{equation*}
 defines a Lagrangian in $\ell^2(\bbZ) \oplus \ell^2(\bbZ)$. Clearly, the composition of two such Lagrangians $L_{\alpha_1}$ and $L_{\alpha_2}$ is the graph of the operator $Q_{\alpha_1}Q_{\alpha_2}$.
 When $\alpha_1$ and $\alpha_2$ have the same sign, then $Q_{\alpha_1} Q_{\alpha_2} = Q_{\alpha_1+\alpha_2}$, so that $L_{\alpha_1} \circ L_{\alpha_2}$ is still a Lagrangian. However, if $\alpha_1$ and $\alpha_2$ have opposite sign, we have $Q_{\alpha_1}Q_{\alpha_2} \subset Q_{\alpha_1+\alpha_2}$, but the composition $Q_{\alpha_1} Q_{\alpha_2}$ is not closed (for example, if $\alpha_1 = - \alpha_2$, then $Q_{\alpha_1} Q_{\alpha_2} \subset \mathrm{id}$ but is not everywhere defined). Therefore, in the case that $\alpha_1$ and $\alpha_2$ have opposite sign, the composition of $L_{\alpha_1}$ and $L_{\alpha_2}$ is not closed, hence not a Lagrangian.
 \end{example}
 
 While in the example above, the composition of $L_{\alpha_1}$ and $L_{\alpha_2}$ may not be closed and therefore not maximal, at least its closure will always be maximal, hence a Lagrangian. However more extreme phenomena are possible:  After conjugating $Q_{\alpha_2}$ by a suitable orthogonal transformation $M$ of $\ell^2(\bbZ)$, the domain of $(M Q_{\alpha_2}M^*) \circ Q_{\alpha_1}$ can even be zero, in other words, the composition of the corresponding Lagrangians is zero.

\begin{remark} \label{RemarkLagrangianSplitting}
If $L_{01} \subset W_0 \oplus - W_1$, we can always orthogonally decompose
\begin{equation*}
  L_{01} = L^\prime_{01} \oplus \bigl(L_0 \oplus \{0\}\bigr) \oplus \bigl(\{0\} \oplus \overline{L}_1\bigr),
\end{equation*}
where $L_0$, $L_1$ are closed isotropic subspaces of $W_0$, respectively $W_1$, and the orthogonal complement $L^\prime_{01}$ is in {\em general position}, meaning that $L^\prime_{01}$ has a trivial intersection with $W_0$ and $W_1$. This implies that $L_{01}^\prime$ is the graph of a densely defined, closed invertible operator $Q: W_0^\prime \rightarrow W_1^\prime$, where $W_i^\prime = (L_i \oplus \overline{L}_i)^\perp$. One easily shows that in order for such a graph to be a Lagrangian, $Q$ must be invertible, and satisfy $Q^* = \overline{Q}^{-1}$. However, the operator may be unbounded in the infinite-dimensional setup, which makes the theory complicated, as Example~\ref{ExampleNonComposable} above shows.
\end{remark}

In the following, we will investigate under which conditions the composition is in fact a Lagrangian. To this end, write 
\begin{equation} \label{WL}
W \coloneqq W_0 \oplus - W_1 \oplus W_1 \oplus -W_2, \qquad L  \coloneqq L_{01} \oplus L_{12},
\end{equation}
 and let $P_L : W \rightarrow L$ be the orthogonal projection. Notice that $\overline{P}_L = P_{\overline{L}}$, the orthogonal projection onto $\overline{L}$.  We moreover set
 \begin{equation}
\begin{aligned} \label{DefinitionULsigmaK} 
  U &\coloneqq \bigl\{ (0, w, w, 0) \in W ~\bigl|~ w \in W_1\bigr\}, \\
  L_\sigma &\coloneqq \bigl\{ (v_0, v_1, v_1, v_2) \in W \mid (v_0, v_1) \in L_{01}, (v_1, v_2) \in L_{12} \bigr\},   \\
  K &\coloneqq \bigl\{ w \in W_1 ~\bigl|~ (0, w) \in L_{01}, (w, 0) \in L_{12}\bigr\}.
\end{aligned}
\end{equation}
We will further use the maps
\begin{equation} \label{DefinitionSigmaDelta}
\begin{aligned}
 \sigma: W &\longrightarrow W_1, ~~~~&  (v_0, v_1, v_1^\prime, v_2) &\longmapsto v_1 - v_1^\prime,\\
 \delta: W_1 &\longrightarrow W, \qquad& w &\longmapsto (0, w, w, 0).
\end{aligned}
\end{equation}
Notice that $\sigma^* = \overline{\delta}$, the conjugate of $\delta$. Clearly, $U = \delta W_1$ and $L_\sigma = \ker(\sigma) \cap L$.

\begin{lemma}
We have the identity
\begin{equation} \label{KernelPLdelta}
\image(\sigma|_L)^\perp = \ker(P_L\overline{\delta}) = \overline{K}.
\end{equation}
\end{lemma}
\begin{proof}
Let $u \in W_1$ with $0 = P_L \overline{\delta}\overline{u} = P_L(0, -\overline{u}, \overline{u}, 0)$. Then $(0, -\overline{u}) \in L_{01}^\perp$ and $(\overline{u}, 0) \in L_{12}^\perp$, hence $(0, u) \in L_{01}$, $(u, 0) \in L_{12}$ and $u \in K$. 
\end{proof}

\begin{lemma} \label{LemmaSplitting}
We have $P_L \overline{U} \perp L_\sigma$, and $P_L \overline{U}$ is dense in $L_\sigma^\perp \cap L$. Hence if $P_L \overline{U}$ and $P_{\overline{L}}U$ are closed, we have the orthogonal splitting
\begin{equation} \label{SplittingOfW}
  W = L_\sigma \oplus P_L \overline{U} \oplus \overline{L}_\sigma \oplus P_{\overline{L}} U.
\end{equation}
\end{lemma}
\begin{proof}
This follows from the calculation
\begin{equation*}
  P_L \overline{U}^\perp \cap L = \image(P_L \overline{\delta})^\perp = \ker\bigl((P_L \overline{\delta})^*\bigr) = \ker\bigl( \sigma P_L^*\bigr) = \ker(\sigma) \cap L = L_\sigma,
\end{equation*}
which uses that $\image(f)^\perp = \ker(f^*)$ for any linear map between inner product spaces. Hence $P_L \overline{U}$ is dense in $L_\sigma^\perp \cap L$. Conjugating this, we get that $P_{\overline{L}} {U}$ is dense in $\overline{L}_\sigma^\perp \cap \overline{L}$. The decomposition \eqref{SplittingOfW} follows.
\end{proof}

\begin{theorem}[Composition] \label{ThmComposition}
Let $W_0, W_1, W_2$ be complex vector spaces with real structures and let $L_{01} \subset W_0 \oplus - W_1$ and $L_{12} \subset W_1 \oplus - W_2$ be Lagrangians. Assume that $W_0, W_1, W_2$ are complete. Then the following statements are equivalent.
\begin{enumerate}
\item[{\normalfont ($i$)}] The map $\sigma$ has closed range when restricted to $L = L_{01} \oplus L_{12}$.
\item[{\normalfont ($ii$)}] The subspaces $P_L \overline{U}$ and $P_{\overline{L}} U$ are closed.
\end{enumerate}
Moreover, when these equivalent statements hold, the composition $L_{02}$ of $L_{01}$ and $L_{12}$ is a Lagrangian in $W_0 \oplus - W_2$.
\end{theorem}

In particular, this shows that in finite dimensions, the composition of two Lagrangians is always a Lagrangian. Moroever, if $L_{01}$ or $L_{12}$ is the graph of a {\em bounded} operator $Q$ with $Q^* = \overline{Q}^{-1}$ (e.g.\ a real unitary), then the composition $L_{02}$ is always a Lagrangian again.

\begin{example}
The converse of the above theorem is {\em not} true: If $W_0 = W_2 = \{0\}$, so that $L_{01}$ and $L_{12}$ are simply Lagrangians in $W_1$, then obviously their composition is $\{0\}$, which is trivially Lagrangian in $W_0 \oplus - W_2 = \{0\}$. In this case condition $(i)$ above is that the sum of $L_{01}$ and $L_{12}$ is closed, which may well be not the case. 

In that setting, the equivalent conditions of Thm.~\ref{ThmComposition} are implied by the statement that the difference $P_{\overline{L}_{01}} - P_{{L}_{12}}$ of orthogonal projections onto $\overline{L}_{01}$, respectively $L_{12}$,  is a compact operator (here the conjugate $\overline{L}_{01}$ appears, as $L_{01}$ by definition is a Lagrangian in $-W_1$). Namely, this in turn implies that $\sigma = P_{L_{01}} - P_{L_{12}}$ is a Fredholm operator, in particular has closed range. 
\end{example}

\begin{proof}
The equivalence $(i) \Leftrightarrow (ii)$, follows from the closed range theorem, which implies that $\image(\sigma|_L) = \image(\sigma P_L^*)$ is closed if and only if $\image((\sigma P_L^*)^*) = \image(P_L \overline{\delta}) = P_L \overline{U}$ is closed.

To see these statements imply that $L_{02}$ is a Lagrangian, we use Thm.~2.1 in \cite{ClosedSubspaces}, which states that for two closed subspaces $X, Y \subseteq W$,
\begin{equation} \label{Theorem2.1}
\text{\itshape{$X+Y$ is closed if and only if $P_{X^\perp} Y$ is closed in $X^\perp$}},
\end{equation}
where $P_{X^\perp}$ is the orthogonal projection onto $X^\perp$ in $W$. Applying this to $X = L$ and $Y = U$, we obtain that the closedness of $P_{\overline{L}}{U}$ implies that $L+U$ is closed. Therefore $\ker(\sigma|_{L+U}) = L_\sigma + U$ is closed as well. Hence using \eqref{Theorem2.1} now for $X = U$ and $Y = L_\sigma$, we get that the closedness of $L_\sigma + U$ implies that $P_{U^\perp} L_\sigma$ is closed. However, since $U^\perp \cong \overline{U}\oplus W_0   \oplus - W_2$ and $\overline{U} \perp L_\sigma$, we have $P_{U^\perp} L_\sigma \cong \{0\} \oplus L_{02}$ so that $L_{02}$ is closed.

So far we know that $L_{02}$ is a closed isotropic subspace. To see that it is maximal, let $(v_0, v_2) \perp L_{02} \oplus \overline{L}_{02}$. Then $(v_0, 0, 0, v_2) \perp L_\sigma \oplus \overline{L}_\sigma$. Hence by the direct sum decomposition \eqref{SplittingOfW}, we have $(v_0, 0, 0, v_2) = P_{\overline{L}}\delta w + P_L \overline{\delta}w^\prime$ for some $w, w^\prime \in W_1$, using that $P_L \overline{U}$ and $P_{\overline{L}} U$ are closed. Hence for all $\ell_{01} \in L_{01}$ and $\ell_{12} \in L_{12}$, we have
\begin{equation*}
\begin{aligned}
\bigl\langle \overline{\ell}_{01}, (v_0, 0)\bigr\rangle + \bigl\langle \overline{\ell}_{12}, (0, v_2)\bigr\rangle &= \bigl\langle (\overline{\ell}_{01}, \overline{\ell}_{12}), P_{\overline{L}}\delta w + P_L \overline{\delta}w^\prime\bigr\rangle\\
&= \bigl\langle (\overline{\ell}_{01}, \overline{\ell}_{12}), P_{\overline{L}}\delta w\bigr\rangle\\
&= \bigl\langle (\overline{\ell}_{01}, \overline{\ell}_{12}), \delta w\bigr\rangle\\
&= \langle \overline{\ell}_{01}, (0, w)\rangle + \langle \overline{\ell}_{12}, (w, 0)\rangle.
\end{aligned}
\end{equation*}
This implies that $(v_0, -w) \in L_{01}$, $(-w, v_2) \in L_{12}$, hence $(v_0, v_2) \in L_{02}$. From the assumption $(v_0, v_2) \perp L_{02} \oplus \overline{L}_{02}$, we get $v_0 = v_2 = 0$, hence $L_{02} \oplus \overline{L}_{02} = W_0 \oplus - W_2$, as $L_{02}$ is closed.
\end{proof}

\begin{remark}[The category of Lagrangian relations] \label{RemarkCategory}
The {\em category of Lagrangian relations} is the category where objects are complex Hilbert spaces with a real structure and morphisms are Lagrangian subspaces. While this idea works fine for finite-dimensional spaces, in the infinite-dimensional situation, one encounters the problem that not all morphisms can be composed, as illustrated by Example~\ref{ExampleNonComposable}. The theorem above allows to repair this by taking {\em polarized} Hilbert spaces, as we discuss now.

To start with, we say that closed subspaces $U$, $V$ of a Hilbert space $W$ are {\em close} if the difference $P_U - P_V$ of the corresponding orthogonal projections is compact. This defines an equivalence relation on the set of closed subspaces of $W$. A {\em sub-Lagrangian} is an isotropic subspace $L$ that is close to a Lagrangian; in other words, there exists a finite-dimensional space $K$ such that $L\oplus K$ is a Lagrangian.

To define the desired category, we start with the objects: These are Hilbert spaces $W$, equipped with an equivalence class $[L]$ of sub-Lagrangians, where $L \sim L^\prime$ if $L$ and $L^\prime$ are close. The set of morphisms between two such objects $(W_1, [L_1])$ and $(W_0, [L_0])$ is the set of Lagrangian subspaces $L_{01} \subset W_0 \oplus - W_1$ such that there exists $L_0^\prime \in [L_0]$ with  $L_0^\prime \subset P_0 L_{01}$ and $L_1^\prime \in [L_1]$ with $\overline{L}_1^\prime \subset P_1 L_{01}$, where $P_i$ is the projection onto $W_i$.

 If $L_{01} \subset W_0 \oplus - W_1$ and $L_{12} \subset W_1\oplus - W_2$ are two such Lagrangians, then there exist $L_1^\prime, L_1^{\prime\prime} \in [L_1]$ such that $\overline{L}_1^\prime \subseteq P_1L_{01}$, $L_1^{\prime\prime} \subseteq P_1L_{12}$, hence
\begin{equation} \label{SigmaLInRem}
  \sigma(L) = P_1 L_{01} + P_1 L_{12} \supseteq \overline{L}^\prime_1 + {L}^{\prime\prime}_1 = \image\bigl( \mathrm{id} + P_{L_1^{\prime\prime}} - P_{L_1^{\prime}}\bigr).
\end{equation}
This is closed since $P_{L_1^{\prime\prime}} - P_{L_1^{\prime}}$ is compact, hence $\mathrm{id} + P_{L_1^\prime} - P_{L_1^{\prime\prime}}$ is a Fredholm operator. Now by condition $(i)$ of Thm.~\ref{ThmComposition}, the composition of $L_{01}$ and $L_{12}$ is a Lagrangian. In fact, \eqref{SigmaLInRem} furthermore shows that $\sigma(L)^\perp$ is finite-dimensional, hence by \eqref{KernelPLdelta}, the space $K$ is always finite-dimensional in this category.

Of course, the condition that $P_U - P_V$ is compact in the definition of closeness can be replaced by the condition that $P_U - P_V \in \mathcal{I}(W)$, where $\mathcal{I}$ is any operator ideal. Using the resulting notion of ``$\mathcal{I}$-closeness'', one obtains subcategories of the above. 
\end{remark}

\subsection{Lagrangian Bimodules and their Tensor Product} \label{SectionBimodules}

Given a Hermitian vector space $W$ with a compatible real structure, we can form the associated Clifford algebra $\Cl(W)$. Moreover, given a Lagrangian $L_{01} \subset W_0 \oplus - W_1$, we will see below that the exterior algebra $\Lambda L_{01}$ is a $\Cl(W_0)$-$\Cl(W_1)$-bimodule in a natural way. This section is devoted to the question in how far this ``second quantization'' procedure gives a functor from the category of polarized Hilbert spaces and Lagrangians discussed in Remark~\ref{RemarkCategory} to the category $\mathrm{sAlg}$ of algebras and bimodules, where objects are algebras and morphisms between algebras $A$, $B$ are $A$-$B$-bimodules, and the composition is given by the tensor product. 

One of the challenges here is that the target category is a bicategory, so that one has to weaken the notion of a functor. More precisely, if $\mathcal{C}$ is a ($1$-)category a functor $T: \mathcal{C} \rightarrow \mathrm{sAlg}$ assigns an algebra $T(x)$ to every object $x$ of $\mathcal{C}$, and an $T(y)$-$T(x)$-bimodule $T(f)$ to every morphism $f: x \rightarrow y$ in $\mathcal{C}$. However with $\mathrm{sAlg}$ being a $2$-category, it would be too restrictive to require strict functoriality; instead, the functor comes with isomorphisms
\begin{equation*}
  \tau_{f, g}: T(g \circ f) \longrightarrow T(g) \otimes_{T(y)} T(f)
\end{equation*}
for any pair of composable morphisms $f: x \rightarrow y$ and $g: y \rightarrow z$. These isomorphisms need to satisfy the coherence condition that the diagram
\begin{equation} \label{CoherenceGeneral}
  \begin{tikzcd}[column sep={3.6cm,between origins}]
  & T(h \circ g \circ f) \ar[dl, "\tau_{h\circ g, f}"', bend right=20] \ar[dr, "\tau_{h, g \circ f}", bend left=20] & \\
  T(h \circ g) \otimes_{T(y)} T(f) \ar[d, "\tau_{h, g} \otimes \mathrm{id}"']& & T(h) \otimes_{T(z)} T(g \circ f) \ar[d, "\mathrm{id} \otimes \tau_{g, f}"] \\
  \bigl(T(h) \otimes_{T(z)} T(g)\bigr) \otimes_{T(y)} T(f) \ar[rr, "\cong"] & & T(h) \otimes_{T(z)} \bigl( T(g) \otimes_{T(y)} T(f)\bigr) 
  \end{tikzcd}
\end{equation}
commute for every triple $f, g, h$ of composable morphisms; here the bottom arrow is the associator of the tensor product.

It turns out that this ``second quantization functor'' is in fact {\em not} a functor, due to the presence of the spaces
\begin{equation*}
  K = \bigl\{ w \in W_1 ~\bigl|~ (0, w) \in L_{01}, (w, 0) \in L_{12} \bigr\},
\end{equation*}
which force the canonical map between $\Lambda L_{02}$ and $\Lambda L_{01} \otimes_{\Cl(W_1)} \Lambda L_{12}$ to be zero, c.f.\ Thm.~\ref{ThmGluing} below. Instead, the tensor product is canonically isomorphic to the twisted module $\Lambda L_{02}\otimes \Lambda^{\mathrm{top}} K$. In this sense, the chiral anomaly has a purely algebraic origin: The non-functoriality of second quantization. 

\medskip

Let $W$ be a complex vector space with a real structure and associated bilinear form $b(\cdot, \cdot)$, as in \eqref{Eq.Bracket}. Let $L \subset W$ be a Lagrangian. The exterior algebra $\Lambda L$ over a Lagrangian $L \subset W$ gives rise to a module over the Clifford algebra $\Cl(W)$ in a natural way, by letting elements of $L$ act via wedging and elements of $\overline{L}$ via insertion. In formulas, we have
\begin{equation*}
\begin{aligned}
  v \cdot \xi &\coloneqq v \wedge \xi, \qquad \overline{v} \cdot \xi \coloneqq \iota(v) \xi,
\end{aligned}
\end{equation*}
for $v \in L$, $\xi \in \Lambda L$, where $\iota(v)$ denotes insertion of $v$ using the inner product. Since $L$ is a Lagrangian, any element $w \in W$ has a unique decomposition $w = v_1 + \overline{v}_2$ with $v_1, v_2 \in L$, hence the above definitions define the action on $W \subset \Cl(W)$ completely. One can then check that the action satisfies the Clifford relations \eqref{CliffordRelationBracket}, hence the action extends to all of $\Cl(W)$ by the universal property of the Clifford algebra.
The module $\Lambda L$ is naturally $\bbZ_2$-graded via its even-and-odd grading, and the action is compatible with the grading, in the sense that even elements of $\Cl(W)$ preserve the grading while odd elements reverse it.

The element $\Omega_L \coloneqq 1 \in \Lambda^0 L \subset \Lambda L$ is called the {\em vacuum vector} of the module $\Lambda L$. It has the property that
\begin{equation*}
  \overline{w}\cdot \Omega_L = 0
\end{equation*}
for all $w \in L$; any element with this property must be a scalar multiple of $\Omega_L$. The following result is standard \cite[\S1.3]{Prat}.

\begin{lemma} \label{LemmaOnHoms}
If $M$ is any $\Cl(W)$-module, then module homomorphisms $\varphi: \Lambda L \rightarrow M$ are in one-to-one correspondence with elements of the {\em Pfaffian Line}
\begin{equation} \label{DefPfaffianLine}
  \mathrm{Pf}(L, M) \coloneqq \bigl\{ m \in M ~\bigl|~ \forall \ell \in L:\overline{\ell} \cdot m = 0\bigr\}.
\end{equation}
More precisely, for any $m \in \mathrm{Pf}(L, M)$, there is a unique homomorphism $\varphi$ such that $\varphi(\Omega_L) = m$ and any homomorphism $\varphi$ is determined by its value on $\Omega_L$. Moreover, if $\varphi(\Omega_L) \neq 0$, then $\varphi$ is injective.
\end{lemma}

Let $W_0$, $W_1$ be Hermitian vector spaces with real structures and let $L_{01} \subset W_0 \oplus -W_1$ be a Lagrangian. By the above considerations, $\Lambda L_{01}$ is a $\bbZ_2$-graded  $\Cl(W_0 \oplus - W_1)$-module. Since 
\begin{equation*}
\Cl(W_0 \oplus - W_1) \cong \Cl(W_0) \otimes \Cl(-W_1) \cong \Cl(W_0) \otimes \Cl(W_1)^{\mathrm{op}}
\end{equation*}
c.f.\ \eqref{OppositeAlgebraEq}, $\Lambda L_{01}$ is equivalently a $\bbZ_2$-graded $\Cl(W_0)$-$\Cl(W_1)$-bimodule. Explicitly, the bimodule structure is given in terms of the $\Cl(W_0 \oplus -W_1)$-structure via
\begin{equation}\label{BimoduleModuleStructure}
\begin{aligned}
w_0 \cdot \xi  &\coloneqq (w_0, 0) \cdot \xi, \qquad \xi \cdot w_1 &\coloneqq (-1)^{|\xi|} (0, w_1) \cdot \xi
\end{aligned}
\end{equation}
for homogeneous elements $\xi \in \Lambda L_{01}$ and $w_i \in W_i$; here the sign comes  from the convention \eqref{SignInOppMult}.
Since the vacuum vector $\Omega_{01}$ of $\Lambda L_{01}$ is annihilated by $\overline{L}_{01}$, we have
\begin{equation} \label{LagrangianSwap}
    0 = \overline{(w_0, w_1)} \cdot \Omega_{01} = (\overline{w}_0, -\overline{w}_1) \cdot \Omega_{01} = \overline{w}_0\cdot\Omega_{01} - \Omega_{01}\cdot \overline{w}_1
  \end{equation}
for all $(w_0, w_1) \in L_{01}$. 

\medskip

\begin{theorem}[Gluing]\label{ThmGluing}
Let $W_0$, $W_1$, $W_2$ be three Hermitian vector spaces with real structure, let $L_{01} \subset W_0 \oplus - W_1$, $L_{12} \subset W_1 \oplus - W_2$ be two Lagrangians and let $L_{02}$ be their composition. Suppose that the space $K$ defined in \eqref{DefinitionULsigmaK} has finite dimension $n$. In the notation of \S\ref{SectionLagrangians}, assume moreover
\begin{enumerate}
\item[{\normalfont (1)}] the map $\sigma$  has closed range when restricted to $L$;
\item[{\normalfont (2)}] the spaces $P_L \overline{U}$ and $P_{\overline{L}} U$ are closed;
\item[{\normalfont (3)}] the composition $L_{02}$ is a Lagrangian in $W_0 \oplus - W_2$.
\end{enumerate}
Then there exists a unique module isomorphism
\begin{equation*}
\begin{aligned}
\alpha: \Lambda L_{02} \otimes {\Lambda^{\Top} K} &\longrightarrow \Lambda L_{01} \otimes_{\Cl(W_1)} \Lambda L_{12}\\
\text{such that} \qquad\Omega_{02} \otimes u_1 \wedge \ldots \wedge u_n &\longmapsto \Omega_{01} \cdot {u}_1 \dots \cdot {u}_n \otimes \Omega_{12}.\qquad\qquad\qquad
\end{aligned}
\end{equation*}
\end{theorem}

Remembering the definition of the Pfaffian line $\mathrm{Pf}(L, M)$  from \eqref{DefPfaffianLine}, the above theorem in particular provides a canonical isomorphism
\begin{equation*}
  \Lambda^{\mathrm{top}} K \cong \mathrm{Pf}\bigl(L_{02}, \Lambda L_{01} \otimes_{\Cl(W_1)} \Lambda L_{12} \bigr).
\end{equation*}
Notice that if each of the spaces $W_0$, $W_1$ and $W_2$ is complete, Thm.~\ref{ThmComposition} implies that the conditions (1) and (2) are equivalent and in fact imply that $L_{02}$ is a Lagrangian. In particular, notice that any two composable morphisms in the category of Lagrangian relations described in Remark~\ref{RemarkCategory} satisfy the assumptions of the theorem. However, since we plan to apply this result to non-complete spaces in the sequel, we state it in this more general version. Our proof requires condition (1) for the surjectivity and condition (2) for the injectivity of $\alpha$.


\begin{proof} 
The proof consists of three steps: The first is well-definedness of the map $\alpha$, where we have to show that it exists and moreover is uniquely determined by its value on a single element. The next two steps are surjectivity and injectivity.

\medskip

\noindent {\em Step 1: Well-definedness.} We start by showing
\begin{equation} \label{ClaimB}
\bigl\{ \Omega_{01} \cdot {u}_1 \cdots {u}_n \otimes \Omega_{12} ~\bigl|~ u_1, \dots, u_n \in K \bigr\} \subseteq \mathrm{Pf}\bigl(L_{02}, \Lambda L_{01} \otimes_{\Cl(W_1)} \Lambda L_{12}\bigr),
\end{equation}
We first show that if $k < n$, then for all  $u_1, \dots, u_k \in K$ and $w_1, \dots, w_m \in K^{\perp}$ with $m$ arbitrary, we have
\begin{equation} \label{ClaimA}
  \Omega_{01} \cdot {w}_1 \cdots {w}_m\cdot{u}_1 \cdots {u}_k \otimes \Omega_{12} = 0. 
\end{equation}
To this end, first observe that for any $v \in W_1$, we have $v\cdot\overline{v} + \overline{v}\cdot v = b(v, \overline{v}) = \|v\|^2$. Since $k<n$, we can find $v \in K$ with $v \perp u_1, \dots, u_k$ and $\| v \|^2 = 1$. Since $v \in K$, we also have $v \perp {w}_1, \dots, {w}_m$. Hence
\begin{equation*}
\begin{aligned}
  \Omega_{01} \cdot {w}_1 \cdots {w}_m  \cdot {u}_1 \cdots {u}_k \otimes \Omega_{12} &=  \Omega_{01} \cdot {w}_1 \cdots {w}_m \cdot {u}_1 \cdots {u}_k \cdot (v \cdot \overline{v}+ \overline{v}\cdot v)\otimes \Omega_{12}\\
  &=\Omega_{01} \cdot {w}_1 \cdots {w}_m\cdot {u}_1 \cdots {u}_k\cdot v \otimes \overline{v}\cdot \Omega_{12} \\
  &\quad +(-1)^{k+m} \Omega_{01}\cdot \overline{v}\cdot {w}_1 \cdots {w}_m \cdot{u}_1 \cdots {u}_k\cdot v\otimes \Omega_{12},
\end{aligned}
\end{equation*}
where we used that due to $v \perp u_j, w_j$, we have $\overline{v}\cdot u_j + u_j\cdot \overline{v} = \overline{v}\cdot w_j + w_j\cdot \overline{v} = 0$. Now since $v \in K$, we have $(0, v) \in L_{01}$ and $(v, 0) \in L_{12}$. Therefore, the identity \eqref{LagrangianSwap} implies that $\Omega_{01}\cdot \overline{v} = 0$ and $\overline{v}\cdot \Omega_{12} = 0$. This shows \eqref{ClaimA}.

To show that $\Omega_{01} \cdot {u}_1 \cdots {u}_n \otimes \Omega_{12}$ is annihilated by $\overline{L}_{02}$, let $(w_0, w_2) \in L_{02}$. By the definition of $L_{02}$, there exists $w_1 \in W_1$ such that $(w_0, w_1) \in L_{01}$ and $(w_1, w_2) \in L_{12}$. Then by \eqref{LagrangianSwap}, we have
\begin{equation*}
\begin{aligned}
  \overline{w}_0\cdot\Omega_{01} \cdot u_1 \cdots {u}_n \otimes \Omega_{12} &= \Omega_{01} \cdot \overline{w}_1\cdot {u}_1 \cdots {u}_n \otimes \Omega_{12}\\
  &= (-1)^n \Omega_{01} \cdot{u}_1 \cdots {u}_n \otimes \overline{w}_1 \cdot \Omega_{12}\\
  &\qquad + \sum_{j=1}^n (-1)^{j} \langle {w}_1, u_j\rangle \Omega_{01}\cdot {u}_1\cdots \widehat{{u}_j}\cdots {u}_n \otimes \Omega_{12}.\\
  &= (-1)^n \Omega_{01} \cdot{u}_1 \cdots {u}_n \otimes \Omega_{12} \cdot \overline{w}_2,
\end{aligned}
\end{equation*}
where we used that by \eqref{ClaimA}, the sum is zero. This shows that $\Omega_{01} \cdot{u}_1 \cdots {u}_n \otimes \Omega_{12}$ is annihilated by $(\overline{w}_0, - \overline{w}_2)$. This finishes the proof of \eqref{ClaimB}.

By \eqref{ClaimA}, the map
\begin{equation*}
  u_1 \wedge \cdots \wedge u_n \longmapsto  \Omega_{01} \cdot {u}_1 \cdots {u}_n \otimes \Omega_{12}
\end{equation*}
is well-defined, since the right hand side is antisymmetric in the entries $u_1, \dots, u_n$: swapping $u_i$ and $u_j$ generates Clifford elements that are products of less than $n$ factors, hence are zero. \eqref{ClaimB} shows that this element is annihilated by $\overline{L}_{02}$, so that $\alpha$ is indeed a well-defined bimodule homomorphism.

\medskip

\noindent {\em Step 2: Surjectivity.} Let us now show that $\alpha$ is surjective, using condition (1). To this end, it suffices to show that elements of the form
\begin{align}\label{Eq.SurjectivityElement}
\Omega_{01} \cdot {w}_1 \cdots {w}_k\cdot  {u}_1 \cdots {u}_n \otimes \Omega_{12}
\end{align}
are in the image of $\alpha$ for $w_1, \dots, w_k \in K^\perp$ and $u_1, \dots, u_n \in K$, as these elements generate $\Lambda L_{01} \otimes_{\Cl(W_1)} \Lambda L_{12}$ as a $\Cl(W_0)$-$\Cl(W_2)$-bimodule. To this end, we claim that there exist elements $a_\nu \in \Cl(W_0)$, $b_\nu \in \Cl(W_2)$, $\nu = 1, \dots, N$ such that 
\begin{align*}
\Omega_{01} \cdot {w}_1 \cdots {w}_k\cdot  {u}_1 \cdots {u}_n \otimes \Omega_{12} =  \sum_{\nu =1}^N a_\nu \cdot  \Omega_{01} \cdot {u}_1 \cdots {u}_n \otimes \Omega_{12} \cdot b_\nu.
\end{align*}
To prove this claim, we proceed by induction. The case $k=0$ is trivial. Now assume that we have proven the claim for all $k^{\prime} < k$. Since $w_1 \in K^\perp = \overline{\image(\sigma)}$ (c.f.\ \eqref{KernelPLdelta}), there are $(v_0, v_1) \in L_{01}$ and $(v^{\prime}_1, v^{\prime}_2) \in L_{12}$ such that $\overline{w}_1 = v_1 - v_1^{\prime}$.  Applying the identity \eqref{LagrangianSwap}, we conclude
\begin{align*}
\Omega_{01} \cdot {w}_1 \cdots {w}_k\cdot u_1 \cdots u_n \otimes \Omega_{12}  &= \Omega_{01}\cdot (\overline{v}_1 - \overline{v}^{\prime}_1)\cdot {w}_2 \cdots {w}_k \cdot u_1 \cdots u_n \otimes \Omega_{12}\\
& = \overline{v}_0\cdot  \Omega_{01} \cdot {w}_2 \cdots {w}_k\cdot u_1 \cdots u_n\otimes \Omega_{12} \\
&\quad -(-1)^{n+k-1} \Omega_{01}  \cdot {w}_2 \cdots {w}_k\cdot u_1 \cdots u_n\otimes \Omega_{12} \cdot \overline{v}_2^{\prime} \\
&\!\!\!\!\!\!\!\!\!\!\!\!\!\!\!  - \sum_{j=2}^k (-1)^{j} b( \overline{v}_1^{\prime}, w_j ) \,\Omega_{01} \cdot {w}_2 \cdots \widehat{{w}_j} \cdots {w}_k\cdot u_1 \cdots u_n\otimes \Omega_{12}\\
&\!\!\!\!\!\!\!\!\!\!\!\!\!\!\!  - \sum_{j=1}^n (-1)^{j+k+1} b( \overline{v}_1^{\prime}, u_j ) \,\Omega_{01} \cdot {w}_2 \cdots {w}_k\cdot u_1 \cdots \widehat{u_j} \cdots u_n\otimes \Omega_{12}
\end{align*}
The last summand is zero by \eqref{ClaimA}, which finishes the induction step and therefore proves the claim.

\medskip

\noindent {\em Step 3: Injectivity.} It follows from Lemma~\ref{LemmaOnHoms} that $\alpha$ is injective if and only if the element $\Omega_{01} \cdot u_1 \cdots u_n \otimes \Omega_{12}$ is a non-zero element of the tensor product. Since we have seen above that this element generates $\Lambda L_{01} \otimes_{\Cl(W_1)} \Lambda L_{12}$ as a bimodule, this amounts to showing that the tensor product is not the trivial module. Using condition (2), we will establish this by showing that there exist non-vanishing linear maps out of it, following along the lines of arguments of Stolz and Teichner \cite[\S2.2]{StolzTeichnerElliptic}.

To start with, we observe that since the Clifford algebra $\Cl(W_1)$ is generated by $W_1$, the tensor product over the Clifford algebra can be realized as the quotient of the tensor product $\Lambda L_{01} \otimes \Lambda L_{12}$ over $\bbC$ by the subspace
\begin{equation*}
  Q = \mathrm{span}\bigl\{ \xi \cdot w \otimes \xi^\prime - \xi \otimes w \cdot \xi^\prime ~\bigl|~ w \in W_1, \xi \in \Lambda L_{01}, \xi^\prime \in \Lambda L_{12} \}.
\end{equation*}
On the other hand, in the notation of \S\ref{SectionLagrangians}, we have the isomorphism
\begin{equation} \label{IsoLPieces}
  \Lambda L_{01} \otimes \Lambda L_{12} \cong \Lambda(L_{01} \oplus L_{12})  = \Lambda L,
\end{equation}
which is a $\Cl(W)$-module since $L$ is a Lagrangian in $W$. Under this isomorphism, the subspace $Q$ corresponds to the subspace $\overline{U} \cdot \Lambda L$, using the Clifford action of $\Cl(W)$ on $\Lambda L$: Indeed, under the isomorphism \eqref{IsoLPieces}, the action of $\overline{U}$ is given by
\begin{equation*}
\begin{aligned}
  (0, -{w}, {w}, 0) \cdot \xi \otimes \xi^\prime &= (0, -w) \cdot \xi \otimes \xi^\prime + (-1)^{|\xi|} \xi \otimes (w, 0)\cdot \xi^\prime\\
  &= (-1)^{|\xi|+1} \bigl( \xi \cdot w \otimes \xi^\prime - \xi \otimes w \cdot \xi^\prime \bigr).
\end{aligned}
\end{equation*}
Our goal is therefore to construct a linear functional $\varphi$ on $\Lambda L$ which is not identically zero, but vanishes on $\overline{U}\cdot\Lambda L$. Under the isomorphism \eqref{IsoLPieces}, $\varphi$ then descends to a non-zero linear functional on the quotient $\Lambda L / (\overline{U} \cdot \Lambda L) \cong \Lambda L_{01} \otimes_{\Cl(W_1)} \Lambda L_{12}$, showing that the tensor product is non-zero. To this end, we observe that with a view on  \eqref{SplittingOfW}, we have the factorization $\Lambda L \cong  \Lambda L_\sigma \otimes \Lambda P_L \overline{U}$, which uses that $P_L \overline{U}$ is closed, by assumption (2). Hence under this isomorphism, every element of $\Lambda L$ is a sum of elements of the form
\begin{equation*}
\ell_1 \wedge \cdots \wedge \ell_k \otimes P_L \overline{u}_1 \wedge \cdots \wedge P_L\overline{u}_l
\end{equation*}
for $\ell_1, \dots, \ell_k \ L_{\sigma}$, $u_1, \dots, u_l \in U$. A linear map $\varphi: \Lambda L \rightarrow V$ is therefore equivalently described by a family of multi-linear maps
\begin{equation} \label{ComponentsOfPhi}
 \varphi_{k, l}: \underbrace{ L_\sigma \times \cdots \times L_\sigma}_k \times \underbrace{P_L \overline{U} \times \cdots \times P_L \overline{U}}_l \longrightarrow V,
\end{equation}
which are alternating in the first $k$ and last $l$ entries. On the other hand, the action of $\overline{U}$ on such an element is given by
\begin{equation} \label{ActionOfU}
\begin{aligned}
\overline{u}_1 \cdot \big( \ell_1 \wedge \cdots &\wedge \ell_k \otimes P_L \overline{u}_2 \wedge \cdots \wedge P_L\overline{u}_l \big) =(-1)^k\ell_1 \wedge \cdots \wedge \ell_k \otimes P_L \overline{u}_1 \wedge \cdots \wedge P_L\overline{u}_l \\
&\quad +\sum_{j=1}^k (-1)^{j-1} \langle P_L {u}_1, \ell_j\rangle \ell_1 \wedge \cdots \widehat{\ell_j} \cdots \wedge \ell_k \otimes P_L \overline{u}_2 \wedge \cdots \wedge P_L\overline{u}_l\\
&\quad+ \sum_{j=2}^l (-1)^{k+j} \langle P_L {u}_1, P_L \overline{u}_j \rangle \ell_1 \wedge \cdots \wedge \ell_k \otimes P_L \overline{u}_2 \wedge \cdots \widehat{P_L \overline{u}_j} \cdots \wedge P_L\overline{u}_l,
\end{aligned}
\end{equation}
where due to $P_L \overline{u}_1 \perp \Lambda L_\sigma$, no term appears that is a wedge product of $k+1$ elements of $L_\sigma$. Rearranging \eqref{ActionOfU}, it follows that in order to ensure $\varphi(\overline{U} \cdot \Lambda L) = 0$, the components $\varphi_{k, l}$ need to satisfy the relation
\begin{equation} \label{TheRelation}
\begin{aligned}
\varphi_{k, l}&\bigl(\ell_1 \wedge \cdots \wedge \ell_k \otimes P_L \overline{u}_1 \wedge \cdots \wedge P_L\overline{u}_l\bigr) \\
&=\sum_{j=1}^k (-1)^{k+j} \langle P_L {u}_1, \ell_j\rangle\varphi_{k-1, l-1}\bigl(\ell_1 \wedge \cdots \widehat{\ell_j} \cdots \wedge \ell_k \otimes P_L \overline{u}_2 \wedge \cdots \wedge P_L\overline{u}_l\bigr)\\
&\quad+ \sum_{j=2}^l (-1)^{j-1} \langle P_L {u}_1, P_L \overline{u}_j \rangle \varphi_{k, l-2} \bigl(\ell_1 \wedge \cdots \wedge \ell_k \otimes P_L \overline{u}_2 \wedge \cdots \widehat{P_L \overline{u}_j} \cdots \wedge P_L\overline{u}_l\bigr).
\end{aligned}
\end{equation}
Suppose first that $K=0$; then by \eqref{KernelPLdelta}, the element $u_1$ is uniquely determined by $P_L \overline{u}_1$ and the relation \eqref{TheRelation} can be used recursively to define the multi-linear maps, starting from an arbitrary family of alternating $k$-linear maps $\varphi_{k, 0}$ on $L_\sigma$. Using the calculation
\begin{equation*}
\begin{aligned}
  \langle P_L w_1, P_L \overline{w}_2\rangle &= \overline{\langle P_{\overline{L}} \overline{w}_1, P_{\overline{L}} {w}_2\rangle} \\
  &= \langle P_{\overline{L}} {w}_2, P_{\overline{L}} \overline{w}_1\rangle \\
  &= \langle (1-P_{{L}}) {w}_2, (1-P_{{L}}) \overline{w}_1\rangle\\
  &= \langle {w}_2, \overline{w}_1\rangle - \langle P_L{w}_2, P_L\overline{w}_1\rangle \\
  &= - \langle P_L{w}_2, P_L\overline{w}_1\rangle
\end{aligned}
\end{equation*}
for $w_1, w_2 \in U$, it is a combinatorial exercise to show that the maps $\varphi_{k, l}$ thus defined iteratively are in fact alternating in the last $l$ entries also, and hence indeed are the components \eqref{ComponentsOfPhi} of a linear functional $\varphi$ (which then is non-vanishing if not all of the $\varphi_{k, 0}$ were vanishing). 

In case that $K\neq 0$, the element $u_1$ is determined by $P_L \overline{u}_1$ only up to an element of $\delta K = \{(0, u, u, 0) \mid u \in K\}$. This forces additional conditions on the intitial maps $\varphi_{k, 0}$; it is not hard to work out that these conditions are precisely the requirement that $\varphi_{k, 0}(\ell_1, \dots, \ell_k) = 0$ unless at least $n$ of the $\ell_j$ are contained in $\delta K$.
\end{proof}

\begin{remark}
Suppose for simplicity that $K=0$ in the last part of the proof above. Then the projection onto $W_0 \oplus - W_2 \subseteq W$ induces an isomorphism $P_{02}: L_\sigma \rightarrow L_{02}$. Hence we may use $\varphi_{k, 0} := \Lambda^k P_{02}$ as initial components for the construction of the linear map $\varphi$ above. Defining $\varphi_{k, l}$ iteratively via \eqref{TheRelation} and passing to the quotient map, one therefore obtains a {\em vector space isomorphism} 
\begin{equation*}
\varphi: \Lambda L/(\overline{U}\cdot \Lambda L) \rightarrow \Lambda L_{02},
\end{equation*}
which sends (the class of) $\Omega_L$ to $\Omega_{02}$. On the other hand, restricting the $\Cl(W)$-action to $W_0 \oplus -W_2$, $\Lambda L$ is naturally a $\Cl(W_0)$-$\Cl(W_2)$-bimodule and since $\overline{U}$ graded commutes with $\Cl(W_0 \oplus -W_2)$, this bimodule structure passes to the quotient $\Lambda L/(\overline{U}\cdot \Lambda L)$. However, one can check that the map $\varphi$ defined before is {\em not} a bimodule homomorphism; in particular, the map $\varphi$ thus constructed is {\em not} the inverse of the canonical isomorphism $\alpha$.
\end{remark}

%
%

To close this section, we discuss coherence of the isomorphisms $\alpha$ defined above. It turns out that they do {\em not} fit into a diagram of the form \eqref{CoherenceGeneral}, due to the presence of the spaces $K$. 

Let $W_0, W_1, W_2$ and $W_3$ be four Hermitian vector spaces with a compatible real structure and let $L_{01}$, $L_{12}$, $L_{23}$ be Lagrangians between them. Assume that all possible (iterated) compositions of these are again Lagrangians (denoted by $L_{ij}$ for $0 \leq i < j \leq 3$), and that for each such composition, the conditions (1)--(2) of Thm.~\ref{ThmGluing} are satisfied. Moreover, assume that each of the spaces
\begin{equation}
\label{DefinitionKandM}
K_{ijk} \coloneqq \bigl\{ u \in W_j \mid (0, u) \in L_{ij}, (u, 0) \in L_{jk} \bigr\}, \qquad 0 \leq i < j < k \leq 3, 
\end{equation}
is finite-dimensional. Then from Thm.~\ref{ThmGluing}, we obtain isomorphisms
\begin{equation*}
\alpha_{ijk}: \Lambda L_{ik} \otimes\Lambda^{\mathrm{top}} K_{ijk} \longrightarrow \Lambda L_{ij} \otimes_{\Cl(W_j)} \Lambda L_{jk}, \qquad 0 \leq i < j < k \leq 3, 
\end{equation*}
Notice that for $u \in W_1$, $(u, 0) \in L_{12}$ implies that also $(u, 0) \in L_{13}$. Hence $K_{012} \subset K_{013}$ and similarly $K_{123} \subset K_{023}$. In particular, we have the orthogonal decompositions
\begin{equation} \label{DirectSumKijk}
K_{013} = K_{012} + K_{012}^\perp\cap K_{013} \quad \text{and} \quad K_{023} = K_{123} + K_{123}^\perp\cap K_{023}.
\end{equation}

\begin{lemma} \label{LemmaIsoOfComplement}
  The relation $\mathcal{R}$ given by $\mathcal{R} \coloneqq L_{12} \cap (K_{012}^\perp\cap K_{013} \oplus K_{123}^\perp\cap K_{023})$  is the graph of a vector space isomorphism
  \begin{equation*}
    \rho_0 : K_{012}^\perp\cap K_{013} \longrightarrow K_{123}^\perp\cap K_{023}.
  \end{equation*}
By definition, $\rho_0(u_1) = u_2$ for $u_1 \in K_{012}^\perp\cap K_{013}$, where $u_2 \in W_2$ is the unique element in $K_{123}^\perp\cap K_{023}$ such that $(u_1, u_2) \in L_{12}$.
\end{lemma}

\begin{proof}
Let $u_1 \in K_{012}^\perp\cap K_{013}$.  Because $K_{012}^{\perp} \cap K_{013}\subset K_{013}$, we have $(0, u_1) \in L_{01}$ and $(u_1, 0) \in L_{13}$.
The latter means that there exists $u_2 \in W_2$ such that $(u_1, u_2) \in L_{12}$ and $(u_2, 0) \in L_{23}$. 
Clearly, $u_2 \in K_{023}$. 
If $u_2^\prime$ is another such element, then $(0, w_2 - u_2^\prime) \in L_{12}$ and $(u_2 - u_2^\prime, 0) \in L_{23}$, hence $u_2-u_2^\prime \in K_{123}$. 
Conversely, modifying $u_2$ by an element of $K_{123}$ clearly does not change the properties $(u_1, u_2) \in L_{12}$ and $(u_2, 0) \in L_{23}$.
Therefore, there exists a unique choice $u_2 \in K_{123}^\perp$ with these properties.

Reversing the argument shows that $\mathcal{R}$ is the graph of a bijection. Since $\mathcal{R}$ is a linear relation, the bijection is a linear map.
\end{proof}

The map $\rho_0$ defined in the lemma above can be extended to a vector space isomorphism
\begin{equation} \label{DefinitionIsoPhi}
\begin{aligned}
  \rho: K_{013} \oplus K_{123} &\longrightarrow K_{012} \oplus K_{023}, \\ 
  (u_1 + v_1^\perp, u_2)& \longmapsto \bigl(u_1, u_2 +\rho_0(v_1^\perp)\bigr),
\end{aligned}
\end{equation}
where the definition is with respect to the decomposition \eqref{DirectSumKijk}.  This induces an isomorphism
\begin{equation*}
  \Lambda^{\mathrm{top}}\rho : \Lambda^{\mathrm{top}}(K_{013} \oplus K_{123}) \longrightarrow \Lambda^{\mathrm{top}}(K_{023} \oplus K_{012})
\end{equation*}
of determinant lines.

\begin{theorem} \label{ThmNumber}
The diagram  
\begin{equation*}
\begin{tikzcd}[column sep={8.7cm,between origins}, row sep={1.8cm,between origins}]
\Lambda L_{03} \otimes \Lambda^{\mathrm{top}} (K_{013} \oplus K_{123})
	\ar[r, " \mathrm{id}\otimes \frac{\Lambda^{\mathrm{top}}\rho}{|\det(\rho^*\rho)}"] 
	\ar[d, "\cong"'] 
	& 
	 	\Lambda L_{03} \otimes \Lambda^{\mathrm{top}} (K_{012} \oplus K_{023}) 
			\ar[d, "\cong"]
	 		\\
\bigl(\Lambda L_{03} \otimes \Lambda^{\mathrm{top}}K_{013}\bigr) \otimes \Lambda^{\mathrm{top}} K_{123} 
	\ar[d, "\alpha_{013} \otimes \mathrm{id}"'] 
	&  
		\bigl(\Lambda L_{03} \otimes \Lambda^{\mathrm{top}}K_{023} \bigr)\otimes \Lambda^{\mathrm{top}} K_{012} 
			\ar[d, "\alpha_{023} \otimes \mathrm{id}"] 
			\\
\bigl(\Lambda L_{01}\otimes_{\Cl(W_1)} \Lambda L_{13}\bigr) \otimes \Lambda^{\mathrm{top}} K_{123} 
	\ar[d, "\cong"'] 
		&  
		\bigl(\Lambda L_{02} \otimes_{\Cl(W_2)} \Lambda L_{23}\bigr) \otimes \Lambda^{\mathrm{top}}K_{012} 
			\ar[d, "\cong"]
			\\
\Lambda L_{01}\otimes_{\Cl(W_1)} \bigl(\Lambda L_{13} \otimes \Lambda^{\mathrm{top}} K_{123} \bigr)
	\ar[d, "\mathrm{id}\otimes\alpha_{123}"'] 
	 	& 
		\bigl(\Lambda L_{02} \otimes \Lambda^{\mathrm{top}}K_{012}\bigr)\otimes_{\Cl(W_2)} \Lambda L_{23}  
			\ar[d, "\alpha_{012}\otimes\mathrm{id}"]
			\\
\Lambda L_{01} \otimes_{\Cl(W_1)} \bigl(\Lambda L_{12} \otimes_{\Cl(W_2)} \Lambda L_{23} \bigr)\ar[r, "\cong"] 
	& \bigl( \Lambda L_{01} \otimes_{\Cl(W_1)} \Lambda L_{12} \bigr) \otimes_{\Cl(W_2)} \Lambda L_{23}
	\end{tikzcd}
\end{equation*}
commutes.
\end{theorem}

\begin{lemma}
\label{LemmaCompleteToBarL}
For each $v \in K_{123}^\perp \cap K_{023}$, there exist $x \in K_{013}^\perp$, $y \in K_{023}^\perp$ such that 
\[
(-\rho_0^*(v) - x, v + y) \in \overline{L}_{12} 
\]
\end{lemma}

\begin{proof}
Set $K := K_{013} \oplus K_{023}$. 
Consider the operator $T : \overline{L}_{12} \to K$ given by orthogonal projection.
Its adjoint $T^*$ is the orthogonal projection from $K$ onto $\overline{L}_{12}$, hence 
\[
\ker(T^*) = K \cap L_{12} = K_{012} \oplus K_{123} \oplus \mathrm{graph}(\rho_0).
\]
We therefore obtain that
\begin{equation*}
\mathrm{ran}(T) = \ker(T^*)^\perp = \mathrm{graph}(- \rho_0^*)
\end{equation*}
In other words, any element of the form $(-\rho_0^*(v), v)$ lies in the image of $T$.
We finish the proof by choosing a preimage and letting $(-x, y)$ be its $K^\perp$ component.
\end{proof}

\begin{proof}[of Thm.~\ref{ThmNumber}]
Let $u_1, \dots, u_l \in K_{012}$, $v_1, \cdots, v_ m\in K_{012}^\perp \cap K_{013}$ and $w_1, \cdots, w_n \in K_{123}$ be orthonormal bases and set
\begin{equation*}
  \Theta \coloneqq u_1 \wedge \cdots \wedge u_l \wedge v_1 \wedge \cdots \wedge v_m \otimes w_1 \wedge \cdots \wedge w_n,
\end{equation*}
a non-zero element of $\Lambda^{\mathrm{top}} K_{013} \otimes \Lambda^{\mathrm{top}} K_{123} \cong \Lambda^{\mathrm{top}} (K_{013} \oplus K_{123}) $. 
Then $\Lambda L_{03} \otimes \Lambda^{\mathrm{top}}(K_{013} \oplus K_{123})$ is generated as a bimodule by $\Omega_{03} \otimes \Theta$. Therefore, the constant $\lambda$ may be computed by computing both compositions on this vector.
The down--right composition sends 
\begin{equation}
\label{downright}
\Theta~~~ \longmapsto~~~\Omega_{01} \cdot u_1 \cdots u_l \cdot v_1 \cdots v_m \otimes \Omega_{12} \cdot w_1 \cdots w_n \otimes \Omega_{23},
\end{equation}
while the right--down composition sends 
\begin{equation}
\label{rightdown}
\Theta~~~ \longmapsto~~~\Omega_{01} \cdot u_1 \cdots u_l \otimes \Omega_{12} \cdot v_1^\prime \cdots v_m^\prime \cdot w_1 \cdots w_n \otimes \Omega_{23},
\end{equation}
where $v_j' = \rho (\rho^*\rho)^{-1}(v_j)$.
We we have to show that these two vectors are, in fact, equal.

By Lemma~\ref{LemmaCompleteToBarL}, for each $j=1, \dots, m$, there exist $x_j \in K_{013}^\perp$ and $y_j \in K_{023}$ such that $(-v_j - x_j, v_j' + y_j) \in \overline{L}_{12}$.
By \eqref{LagrangianSwap}, this means that
\begin{equation}
\label{SwapOver}
(v_j + x_j) \cdot \Omega_{12} = \Omega_{12} \cdot (v_j' + y_j).
\end{equation}
Here we used that $v_j = \rho^*(v_j') = \rho_0^*(v_j')$.
Using \eqref{SwapOver} $m$ times, the theorem therefore follows if we can prove that the right hand side of \eqref{downright} equals
\begin{equation}
\label{Replacement}
\Omega_{01} \cdot u_1 \cdots u_l \cdot (v_1 + x_1) \cdots (v_m + x_m) \otimes \Omega_{12} \cdot w_1 \cdots w_n \otimes \Omega_{23}
\end{equation}
and similarly, that on the right hand side of \eqref{rightdown}, we may replace $v_j'$ with $v_j' + y_j$ without changing the result.
Now, \eqref{Replacement} may be expanded as 
\[
\sum_{a=0}^m \sum_{i_1 < \dots < i_a} \epsilon_{i_1, \dots, i_a} \cdot \Omega_{01} \cdot u_1 \cdots u_l \cdot x_{i_1} \cdots x_{i_a} \cdot v_1 \cdots \widehat{v_{i_1}} \cdots \widehat{v_{i_a}} \cdots v_{n} \otimes \Omega_{12} \cdot w_1 \cdots w_n \otimes \Omega_{23},
\]
where $\epsilon_{i_1, \dots, i_a} \in \{\pm 1\}$ and $\widehat{v_i}$ indicates that the vector $v_i$  is omitted in the product.
The summand for $a=0$ is precisely the right hand side of \eqref{downright}, so we need to show all summands with $a \neq 0$ vanish.
However, writing
\begin{align*}
&\Omega_{01} \cdot u_1 \cdots u_l \cdot x_{i_1} \cdots x_{i_a} \cdot v_1 \cdots \widehat{v_{i_1}} \cdots \widehat{v_{i_a}} \cdots v_{n} \otimes \Omega_{12} \cdot w_1 \cdots w_n \otimes \Omega_{23} 
\\
&= (\mathrm{id} \otimes \alpha_{123})\bigl(\underbracket{\Omega_{01} \cdot u_1 \cdots u_l \cdot x_{i_1} \cdots x_{i_a} \cdot v_1 \cdots \widehat{v_{i_1}} \cdots \widehat{v_{i_a}} \cdots v_{n} \otimes \Omega_{13}} \otimes w_1 \wedge \cdots \wedge w_n \bigr)
\end{align*}
the underbracketed term vanishes by \eqref{ClaimA} if $a\neq 0$, as $u_1, \dots, u_l, x_1, \dots, x_m \perp K_{013}$.
This shows that, indeed \eqref{Replacement} agrees with the right hand side of \eqref{downright}.
For the right hand side of \eqref{rightdown}, one argues similarly and thereby finishes the proof.
\end{proof}


\section{Spin Geometry and Functorial Field Theory} \label{SectionSpinGeometry}

This is the geometric/analytic part of this paper. We start by introducing the class of manifolds we will be working with, then we discuss some features of boundary value problems for Dirac operators needed in the sequel; finally, we end with defining the desired twist functor.

\subsection{Geometry of Manifolds with Clifford Module} \label{SectionCliffordManifolds}

While the results  in the introduction where stated for spin manifolds and corresponding bordisms, it is useful to formulate our results in a more general context.

\begin{definition}[Clifford manifolds]
Fix a dimension $d \geq 0$ and a codimension $0 \leq k \leq d$. A {\em $d$-dimensional Clifford module} on a $(d-k)$-dimensional Riemannian manifold $Y$ (possibly with boundary) is a $\bbZ_2$-graded Hermitian vector bundle $\Sigma_Y$ with compatible connection, together with a parallel endomorphism field
\begin{equation*}
  \gamma: TY \oplus \underline{\bbR}^k \longrightarrow \mathrm{End}^-(\Sigma_Y)
\end{equation*}
satisfying the Clifford relations
\begin{equation} \label{OtherCliffordRelations}
   \gamma(v) \gamma(w) + \gamma(w)\gamma (v) = -2\langle v, w\rangle,
\end{equation}
for $v, w \in TY \oplus \underline{\bbR}^k$. Here $\underline{\bbR}^k$ denotes the $k$-dimensional trivial line bundle over $Y$. We will usually just say {\em Clifford module} if the dimension $d$ is fixed throughout and no confusion is likely to arise. A Riemannian manifold $Y$ with a ($d$-dimensional) Clifford module will be referred to as a {\em Clifford manifold}; we will call $\Sigma_Y$ the {\em spinor bundle} on $Y$. 

An {\em isomorphism} of two Clifford manifolds $(Y, \Sigma_Y, \gamma_Y)$ and $(Y^\prime, \Sigma_{Y^\prime}, \gamma_{Y^\prime})$ consists of an isometry $f: Y \rightarrow Y^\prime$ that is covered by a bundle isomorphism $F: \Sigma_Y \rightarrow \Sigma_{Y^\prime}$ which preserves the grading, the connection and the Hermitian inner product, and that intertwines the Clifford maps, in the sense that
\begin{equation*}
   \gamma_{Y^\prime}\bigl(df(w)\bigr)F(\psi) = F\bigl( \gamma_Y(w) \psi\bigr)
\end{equation*}
for all $w \in TY\oplus \underline{\bbR}^k$ and $\psi \in \Sigma_Y$.
\end{definition}

\begin{remarks}
Let us make the following comments on the definition above.
\begin{enumerate}[(1)]
\item The endomorphism field $\gamma$ turns $\Sigma_Y$ into a bundle of graded modules for the algebra bundle $\Cl(TY \oplus \underline{\bbR}^k)$. 
\item That $\Sigma_Y$ is $\bbZ_2$-graded means that it splits as a direct sum $\Sigma_Y = \Sigma_Y^+ \oplus \Sigma_Y^-$, where the summands are orthogonal with respect to the Hermitian scalar product. Since the connection is compatible, this implies that it preserves the summands. The operator
\begin{equation} \label{DefinitionJ}
  J ~~\widehat{=}~~ \begin{pmatrix} i & 0 \\ 0 & -i\end{pmatrix} \qquad \text{w.r.t} \qquad \Sigma_Y = \Sigma_Y^+\oplus \Sigma_Y^-
\end{equation}
is a complex structure and will be referred to as the {\em grading operator}.
\item An endomorphism of $\Sigma_Y$ is even, respectively odd, if it preserves, respectively exchanges, the summands $\Sigma_Y^\pm$. In particular, we require $\gamma(v)$ to be odd for each $v \in TY \oplus \underline{\bbR}^k$. That $\gamma$ is parallel means equivalently that we have the product rule 
\begin{equation*}
  \nabla_v^\Sigma\bigl(\gamma(w) \psi\bigr) = \gamma(\nabla_v w)\psi + \gamma(w) \nabla^\Sigma_v \psi,
\end{equation*}
for vector fields $v, w \in C^\infty(Y, TY \oplus \underline{\bbR}^k)$ and $\psi \in C^\infty(Y, \Sigma_Y)$, where $\nabla^\Sigma$ is the connection on $\Sigma_Y$ and $\nabla$ denotes the  Levi-Civita-connection on $TY \oplus \underline{\bbR}^k$.
\end{enumerate}
\end{remarks}

\begin{examples}
Clifford manifolds arise in the following more special situations, where we fix the dimension $d$ for all Clifford modules.
\begin{enumerate}[(1)]
\item A spin structure on a $d$-dimensional Riemannian manifold $X$, $d$ even, gives rise to a canonical spinor bundle $\Sigma_X$, which is a graded $\Cl(TX)$-module. More generally, if a $(d-k)$-dimensional manifold $Y$ is equipped with principal $\mathrm{Spin}_d$-bundle $P$ lifting the structure group of the frame bundle $\mathrm{Fr}(TY \oplus \underline{\bbR}^k)$, one can form the associated bundle $\Sigma_Y \coloneqq P \times_{\mathrm{Spin}_d} \Sigma$ where $\Sigma$ is the spin representation. This admits a canonical Clifford map, as the representation $\Sigma$ of $\mathrm{Spin}_d$ is naturally a graded $\Cl_d$-module.
\item Similar remarks hold when a spin structure is replaced by a spin$^c$ structure. The difference is that in the case of a spin structure, the Levi-Civita connection determines a canonical connection compatible with the metric and Clifford action, while in the spin$^c$ case, one has a certain freedom in the choice of a connection.
\item Any $(d-k)$-dimensional Riemannian manifold $Y$ admits a canonical Clifford module $\Sigma_Y = \Lambda (TY \oplus \underline{\bbR}^k)\otimes \bbC$, the complexified exterior algebra of $TY \oplus \underline{\bbR}^k$. The Clifford map is given by
\begin{equation*}
  \gamma(w) \xi \coloneqq \frac{\sqrt{2}}{2}\bigl(w \wedge \xi - i(w) \xi\bigr),
\end{equation*}
for $\xi \in \Lambda TY \oplus \underline{\bbR}^k$, where $i(w)$ denotes insertion of $w \in TY \oplus \underline{\bbR}^k$. Reduction of the form degree modulo $2$ gives a $\bbZ_2$-grading on $\Sigma_Y$. If $Y$ is oriented, there is another possible choice of grading, given by splitting $\Sigma_Y$ into the eigenspaces of the Hodge-$*$-operator.
\item Manifolds with a reduction of their structure group to one of the ten symmetry groups of Freed and Hopkins can also be turned into Clifford manifolds, c.f.\ \cite[\S9]{FreedHopkins}.
\item All examples above generalize to manifolds {\em over a target}, in the spirit of supersymmetric $\sigma$-models. After fixing a manifold $M$ and a Hermitian vector bundle $E$ with compatible connection over $M$ (the {\em target}), any $(d-k)$-dimensional manifold $Y$ with Clifford module $\Sigma_Y$ together with a map $\rho: Y \rightarrow M$ gives rise to a new (still $d$-dimensional) twisted Clifford module $\Sigma_{Y, \rho} \coloneqq \Sigma_Y \otimes \rho^*E$, where the Clifford map acts as the identity on the $\rho^*E$ factor. This gives rise to a {\em bundle} of Clifford manifolds over the mapping space $C^\infty(Y, M)$.
\end{enumerate}
\end{examples}

We observe that a $d$-dimensional manifold $X$ with $d$-dimensional Clifford module induces a $d$-dimensional Clifford module on its boundary $\partial X$ by restriction, using the identification
\begin{equation*}
  TX|_{\partial X} \cong T\partial X \oplus N\partial X \cong T\partial X \oplus \underline{\bbR}, \qquad v + a \nu \longmapsto v + a\mathbf{1}
\end{equation*}
for $v \in T\partial X$, $a \in \bbR$, where our convention is that we identify the {\em outward pointing} normal vector field $\nu \in N \partial X$ with the canonical section $\mathbf{1}$ of $\underline{\bbR}$. Hence the boundary of a $d$-dimensional Clifford manifold is naturally a $(d-1)$-dimensional Clifford manifold.

\begin{definition}[Duals and bordisms] \label{DefinitionDuals}
Fix a dimension $d$ for all Clifford modules. 
\begin{enumerate}[(i)]
\item If $Y$ is a $(d-1)$-dimensional Clifford manifold, its {\em dual} $Y^\vee$ has the same underlying manifold and spinor bundle, but the Clifford map altered by
\begin{equation*}
  \gamma_{Y^\vee}(w) \coloneqq \gamma_Y(\kappa w),
\end{equation*}
for $w \in TY\oplus \underline{\bbR}$, where $\kappa$ is the reflection at the canonical section $\mathbf{1} \in \underline{\bbR} \subseteq TY \oplus \underline{\bbR}$. 
\item Given two closed $(d-1)$-dimensional Clifford manifolds $Y_0$, $Y_1$, a {\em bordism} from $Y_1$ to $Y_0$ is a compact $d$-dimensional Clifford manifold $X$ together with an isomorphism $f: \partial X \rightarrow Y_0 \sqcup Y_1^\vee$ of Clifford manifolds.
\end{enumerate}
\end{definition}

\begin{figure}[h] 
\begin{center}
\includegraphics[scale=0.22, draft = false]{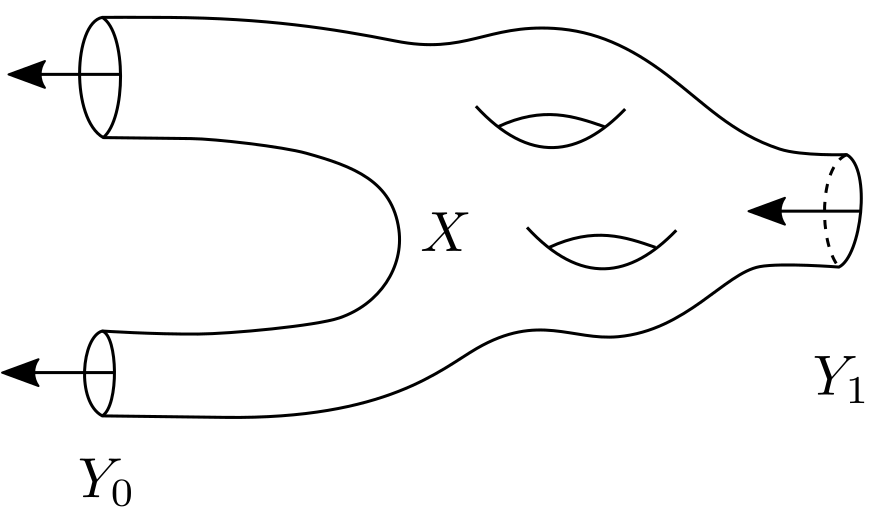}

\parbox{11cm}{\caption{Bordism between Clifford manifolds with arrows of time, given by normal vectors.}} \label{FigureBordism}
\end{center}
\end{figure}

\begin{example}[Cylinders]
If $Y$ is a $(d-1)$-dimensional manifold with $d$-dimensional Clifford module, this data induces a $d$-dimensional Clifford module on the Riemannian manifold $Y \times I$ (where $I \subset \bbR$ is any interval) called a {\em cylinder} over $Y$. Namely, we set $\Sigma_{Y \times I} = \Sigma_Y$, and after identifying
\begin{equation*}
  T(Y \times I) = TY \oplus TI = T Y \oplus \underline{\bbR},
\end{equation*}
we can set $\gamma_{Y\times I} = \gamma_Y$. If $I = [a, b]$, a closed interval, $Y\times I$ is a manifold with boundary, with two boundary components, $Y \times \{a\}$ and $Y \times \{b\}$. Since the Clifford module on the boundary is induced by identifying the {\em outward} normal vector with $\mathbf{1} \in \underline{\bbR}$, $Y \times \{a\}$ is canonically isomorphic to $Y$ as a Clifford manifold, while $Y \times \{b\}$ is canonically isomorphic to $Y^\vee$.
\end{example}

The following technical definition uses this notion of a cylinder over $(d-1)$-dimensional Clifford manifolds.

\begin{definition}[Product structure]
We say that the a $d$-dimensional Clifford manifold $X$  has {\em product structure near the boundary}, if there exists a neighborhood $U$ of the boundary $\partial X$ that for some $\varepsilon>0$ is isomorphic to $\partial X \times [0, \varepsilon)$ as a Clifford manifold. 
\end{definition}

%

\begin{remark}[Sketch of the bordism category] \label{RemarkCliffordBordism}
Roughly, the $d$-dimensional bordism category $\mathrm{Bord}_{\langle d-1, d\rangle}^{\Cl}$ can be described as follows. Objects of the category are  $(d-1)$-dimensional Riemannian manifolds $Y$ with $d$-dimensional Clifford module. If $Y_0$, $Y_1$ are two such manifolds, a morphism from $Y_1$ to $Y_0$ is a Clifford bordism $X$ between them that has product structure near the boundary. Composition is given by gluing of bordisms, where the product structure near the boundary ensures that the metric and Clifford structure on the resulting bordism is again smooth.

One reason we leave this definition as a sketch is that $\mathrm{Bord}_{\langle d-1, d\rangle}^{\Cl}$ should actually be a {\em bicategory}, as bordisms may have automorphisms: isometries of Clifford manifolds. However, for our purposes, we do not need a precise definition of the bordism category, which we rather use as a guiding concept. Instead, we subsequently formulate our results in such a way that they are independent of the definition of the bordism category and can be used to rigorously define a functor, once a precise definition of the bordism category is established.
\end{remark}

\begin{remark}
If one defines the {\em dual of a bordism} $X$ between $Y_1$ and $Y_0$ as the obvious bordism $X^\vee$ between $Y_0^\vee$ and $Y_1^\vee$, the notion of dual thus defined should coincide with the category-theoretic notion of dual after rigorously constructing the bordism category; see e.g.\ \cite{PontoShulman}.
\end{remark}

We end this section by discussing an extra structure on the category $\mathrm{Bord}_{\langle d-1, d\rangle}^{\Cl}$ needed subsequently.

\begin{definition}[Conjugate]
  The {\em conjugate} of a $(d-k)$-dimensional Clifford manifold $Y$  is the Clifford manifold $\overline{Y}$ which has the same underlying manifold, spinor bundle and Clifford map, but the grading operator $J$ (c.f.\ \eqref{DefinitionJ}) is replaced by $-J$. In other words, the grading of $\Sigma_{\overline{Y}}$ is the opposite of that of $\Sigma_Y$, that is, $\Sigma_{\overline{Y}}^\pm = \Sigma_Y^\mp$. 
  \end{definition}

The operation sending $Y$ to $\overline{Y}$ and $X$ to $\overline{X}$ is an involution on the category $\mathrm{Bord}_{\langle d-1, d\rangle}^{\Cl}$. An important observation is now that for $(d-1)$-dimensional Clifford manifolds $Y$, there is a canonical isomorphism $(\mathrm{id}, \eta_Y): \overline{Y} \longrightarrow Y^\vee$ given by the identity on the underlying manifold and the bundle isomorphism
\begin{equation} \label{BundleIso}
  {\eta}_Y: \Sigma_{\overline{Y}} \longrightarrow \Sigma_{Y^\vee}, \qquad \psi \longmapsto \gamma_{\overline{Y}}(\mathbf{1})\psi,
\end{equation}
covering it.
Notice that $\Sigma_{\overline{Y}} = \Sigma_{Y^\vee} = \Sigma_Y$ as ungraded vector bundles and $\gamma_{\overline{Y}} = \gamma_{{Y}}$ so that this definition makes sense.  But  the grading of $\Sigma_{\overline{Y}}$ is the opposite of that of $\Sigma_{\overline{Y}}$, hence $ {\eta}_Y$ is grading preserving. That $ \eta_Y$ intertwines the Clifford maps of $\overline{Y}$ and $Y^\vee$ follows from the calculation
\begin{equation*}
 \eta_Y\bigl(\gamma_{\overline{Y}}(w)\psi\bigr) = \gamma_Y(\mathbf{1})\gamma_Y(w) \psi = \gamma_Y(\mathbf{1})\gamma_Y(w) \gamma_Y(\mathbf{1})^{-1}   \eta_Y(\psi) = \gamma_Y(\kappa w)  \eta_Y(\psi),
\end{equation*}
for $w \in TY \oplus \underline{\bbR}$ and $\psi \in \Sigma_{\overline{Y}} = \Sigma_Y$, where $\kappa$ is the reflection at $\mathbf{1} \in \underline{\bbR} \subseteq TY \oplus \underline{\bbR}$. 

If $X$ is a bordism from $Y_1$ to $Y_0$,  then $\overline{X}$ is a bordism from $\overline{Y}_1$ to $\overline{Y}_0$ , or equivalently, between $Y_1^\vee$ and $Y_0^\vee$ after using these identifications $\overline{Y}_i \cong Y_i^\vee$; this in turn is the same thing as a bordism from $Y_0$ to $Y_1$. In particular, if $X$ is a compact Clifford manifold with boundary $\partial X = Y$, seen as a bordism from $Y$ to $\emptyset$, we can form its {\em double} $X \sqcup \overline{X}$, a closed manifold,  provided $X$ has product structure near the boundary.


\subsection{Boundary Analysis of Dirac Operators} \label{SectionAnalysis}

Throughout, fix a dimension $d$ for all Clifford modules. If $X$ is a $d$-dimensional Riemannian manifold with ($d$-dimensional) Clifford module $\Sigma_X$ as introduced in \S\ref{SectionCliffordManifolds}, the Clifford map allows to define a canonical first order differential operator $D_X$ on $X$, the {\em Dirac operator}, given by
\begin{equation*}
  D_X \Phi = \sum_{j=1}^d \gamma(e_j) \nabla_{e_j}^\Sigma \Phi
\end{equation*}
for $\Phi \in C^\infty(X, \Sigma_X)$, with respect to a local orthonormal frame $e_1, \dots, e_d$. The fact that $\gamma$ is parallel implies that the Dirac operator is symmetric on functions with compact support in the interior of $X$. In general, if $X$ is a bordism of Clifford manifolds from $Y_1$ to $Y_0$, one has the {\em integration by parts formula}
\begin{equation} \label{IntegrationByParts}
  \langle D_X \Phi, \Psi\rangle_{L^2(X)} - \langle  \Phi, D_X\Psi\rangle_{L^2(X)} = \bigl\langle \varphi_0, \gamma(\nu) \psi_0\bigr\rangle_{L^2(Y_0)} - \bigl\langle \varphi_1, \gamma(\nu) \psi_1\bigr\rangle_{L^2(Y_1)},
\end{equation}
for $\Phi, \Psi \in C^\infty(X, \Sigma_X)$, where we wrote $\Phi|_{Y_i} = \varphi_i$, $\Psi|_{Y_i} = \psi_i$. Here the signs on the right hand side arise from our conventions regarding the normal vector.  It will be important throughout our considerations that $D_X$ satisfies the following {\em unique continuation property}:
\begin{equation} \label{UCP}
\begin{aligned}
  &\text{\itshape If $X$ is connected and a harmonic spinor $\Phi$ vanishes  }\\
  &\text{\itshape on some non-empty hypersurface $Y$ of $X$, then $\Phi \equiv 0$. }
\end{aligned}
\end{equation}
For references, see \S8 of \cite{BoosBavnbek}; c.f.\ also \cite{BaerNodal}.

\medskip

We denote by
\begin{equation} \label{DefHarmonicSpinors}
  \mathcal{H}_X \coloneqq \bigl\{ \Phi \mid D_X\Phi = 0 \bigr\} \subset C^\infty(X, \Sigma_X)
\end{equation}
 the space of {\em harmonic spinors} on $X$. We remark that if $X$ has a boundary, by a smooth spinor $\Phi$ on $X$, we mean one that is smooth {\em up to the boundary}, meaning, we require that whenever $X$ is embedded in larger open manifold $\tilde{X}$, $\Phi$ has a smooth extension to $\tilde{X}$. An important role in our considerations will be played by the space
 \begin{equation}\label{DefinitionLagrangian}
   L_X \coloneqq \bigl\{ \Phi|_{\partial X} \mid \Phi \in \mathcal{H}_X \bigr\} \subseteq C^\infty(\partial X, \Sigma_{\partial X}).
 \end{equation}
 A direct consequence of the unique continuation property \eqref{UCP} is the following.
 
 \begin{lemma} \label{LemmaKernel}
Let $X$ be a $d$-dimensional compact Clifford manifold. Then the kernel of the restriction map $\mathcal{H}_X \rightarrow L_X$ is $\mathcal{H}_{X^{\cl}}$, the space of harmonic spinors on $X^{\cl} \subseteq X$, the {\em closed part} of $X$, i.e.\ the union of all connected components that do not touch the boundary. 
\end{lemma}


The following result on the spaces $L_X$ defined in \eqref{DefinitionLagrangian} is fundamental to our observations; a proof can be found e.g.\ in \cite[\S XVII, Lemma~B]{Palais}; c.f.\ also \cite[\S12]{BoosBavnbek}.

\begin{theorem} \label{ThmCobordism}
Let $X$ be a $d$-dimensional closed Clifford manifold and assume that the Dirac operator is invertible. Let $Y\subset X$ be a closed hypersurface that divides $X$ into two parts, $X_0$ and $X_1$. Then we have
\begin{equation*}
  L_{X_0} \cap L_{X_1} = \{0\} \qquad \text{and} \qquad C^\infty(Y, \Sigma_Y) = L_{X_0} + L_{X_1}.
\end{equation*}
Moreover, the orthogonal projection $P_0$ onto $L_{X_0}$ along $L_{X_1}$ is a pseudodifferential operator of order zero, with principal symbol $p_0(\xi)$, $\xi \in T^\prime Y$, the orthogonal projection onto the $i \vert \xi\vert$ eigenspace of the endomorphism $\gamma(\nu)\gamma(\xi)$ of $\Sigma_Y$. Here $\nu$ is the normal vector to $Y$ pointing into $X_0$. Consequently, the principal symbol $p_1(\xi)$ of $P_1 = \mathrm{id}-P_0$ is the projection onto the $-i \vert \xi\vert$ eigenspace of $\gamma(\nu)\gamma(\xi)$.
\end{theorem}

We use this to prove the following result.

\begin{theorem} \label{TheoremLagrangian}
Let $X$ be a $d$-dimensional compact Clifford manifold and assume that $X$ has product structure near the boundary. Then 
\begin{equation*}
L_X^\perp =  \gamma(\nu)L_X, 
\end{equation*}
where the orthogonal complement is taken in $C^\infty(\partial X, \Sigma_{\partial X})$ with respect to the $L^2$-scalar product. 
\end{theorem}

\begin{proof}
Since the kernel of the restriction map $\mathcal{H}_X \rightarrow L_X$ is the space of harmonic spinors on the closed part of $X$, c.f.\ Lemma \ref{LemmaKernel}, we assume without loss of generality that $X$ has no closed components.

 Let $M = {X} \sqcup_{\partial X}  \overline{X}$ be the closed Clifford manifold obtained from gluing $X$ together with its conjugate along the common boundary, as explained at the end of \S\ref{SectionCliffordManifolds}.
The obvious isometry $f: M \rightarrow M$ that exchanges $X$ and $\overline{X}$ and fixes the hypersurface $\partial X \subset M$ is covered by a bundle isomorphism $F : \Sigma_M \rightarrow \Sigma_M$ that preserves the Clifford map $\gamma_M$ but is grading reversing, given by the canonical identification $\Sigma_{\overline{X}} = \Sigma_X$ (in other words, $(f, F)$ is an isomorphism of Clifford manifolds from $M$ to $\overline{M}$). On $\partial X$, we have $F|_{\partial X} = \eta_{\partial X}$, the bundle isomorphism used for gluing $X$ to $\overline{X}$. By \eqref{BundleIso}, we therefore have $F|_{\partial X} = \gamma(\nu)$, multiplication by the normal vector $\nu$ that points out of $X$.

After these preparations, we first claim that $\mathcal{H}_M = \{0\}$. Indeed, suppose that $\Phi \in \mathcal{H}_M$, and let $\tilde{\Phi} = F^*\Phi \in \mathcal{H}_M$ the harmonic spinor obtained by reflecting $\Phi$ at $\partial M$, and let $\varphi= \Phi|_{\partial X}$, so that $\tilde{\Phi}|_{\partial X} = \gamma(\nu)\varphi$. Then
\begin{equation*}
  0 = \langle D_X \Phi, \tilde{\Phi}\rangle_{L^2(X)} - \langle  \Phi, D_X\tilde{\Phi}\rangle_{L^2(X)} = \langle \varphi, \gamma(\nu)\tilde{\varphi}\rangle_{L^2(\partial X)} = - \|\varphi\|_{L^2(\partial X)}^2,
\end{equation*}
hence $\varphi = 0$. By the unique continuation property \eqref{UCP}, this implies that also $\Phi = 0$ (as far as we know, this argument is due to \cite[Prop~9.3]{BoosBavnbek}, c.f.\ also \cite{BoosBavnbekLesch}).

Now using Thm.~\ref{ThmCobordism} with $X_0 = X$, $X_1 = \overline{X}$ and $Y = \partial X$, we obtain that $C^\infty(\partial X, \Sigma_{\partial X})$ is the direct sum of $L_{X}$ and $L_{\overline{X}}$. Let $\varphi \in L_X$ and choose $\Phi \in C^\infty(M, \Sigma_M)$ with $D_M \Phi = 0$ on $X$ and $\Phi|_{\partial X} = \varphi$. Then the reflection $\tilde{\Phi} := F^*\Phi$ of $\Phi$ satisfies $D_M \tilde{\Phi}= 0$ on $\overline{X}$ and $\tilde{\Phi}|_{\partial X} = \gamma(\nu)\varphi$, hence $\gamma(\nu)\varphi \in L_{\overline{X}}$; in other words $\gamma(\nu)L_X \subseteq L_{\overline{X}}$. Reversing this argument shows that $L_{\overline{X}} = \gamma(\nu)L_X$.

It is left to show that $\gamma(\nu)L_X \subseteq L_X^\perp$. To see this, let $\varphi, \psi \in L_X$ and choose $\Phi, \Psi \in C^\infty(M, \Sigma_M)$ with $D_M \Phi = D_M \Psi = 0$ on $X$ and $\Phi|_{\partial X} = \varphi$ and $\Psi|_{\partial X} = \psi$. Then
\begin{equation*}
  0 = \langle D_M \Phi, \Psi\rangle_{L^2(X)} - \langle  \Phi, D_M \Psi\rangle_{L^2(X)} = \bigl\langle \varphi, \gamma(\nu)\psi\bigr\rangle_{L^2(\partial X)},
\end{equation*}
which finishes the proof.
\end{proof}

This allows to generalize Thm.~\ref{ThmCobordism} as follows.

\begin{corollary} \label{CorCobordism}
If in Thm.~\ref{ThmCobordism}, we drop the assumption of invertibility of $D_X$, we have
\begin{equation*}
  L_{X_0} \cap L_{X_1} = K, \qquad \text{and} \qquad  L_{X_0} + L_{X_1}= \bigl(\gamma(\nu) K\bigr)^\perp,
\end{equation*}
 where $K \coloneqq \bigl\{ \Phi|_Y  \mid \Phi \in \mathcal{H}_X \bigr\}$.
\end{corollary}

\begin{proof}
Let now $\varphi \perp L_{X_0} +L_{X_1}$. 
Then $\varphi \perp L_{X_0}$, $\varphi \perp L_{X_1}$, hence by Thm.~\ref{TheoremLagrangian}, $\varphi \in \gamma(\nu)L_{X_0}$ and $\varphi \in \gamma(\nu) L_{X_1}$; in other words, $\varphi \in \gamma(\nu)K$. Hence
\begin{equation*}
  K = \gamma(\nu)\bigl(L_{X_0}+L_{X_1}\bigr)^\perp = \gamma(\nu)\bigl(\gamma(\nu) L_{X_0} \cap \gamma(\nu) L_{X_1}\bigr) = L_{X_0} \cap L_{X_1}.
\end{equation*}
It remains to show that $L_{X_0} + L_{X_1}$ is closed. 
To this end, we first prove that $L_{X_0}^c+L_{X_1}^c$ is closed, where $L_{X_i}^c$ is the closure of $L_{X_i}$ in $L^2(Y, \Sigma_Y)$. By \eqref{Theorem2.1}, it suffices to show that the space $P_{0}^\perp L_{X_1}^c$ or, equivalently, the space $P_{1}^\perp L_{X_0}^c$ is closed, where $P_{i}$ is the orthogonal projection onto $L_{X_i}^c$ and $P_i^\perp = 1-P_i$. To this end, it suffices to show that the operator $P_1^\perp P_0 + P_1 P_0^\perp$ has closed range. Now by Thm.~\ref{ThmCobordism} and Thm.~\ref{TheoremLagrangian}, each of the operators $P_i$, $P_i^\perp$ is a pseudodifferential operator of order zero. Moreover, the principal symbols $p_0(\xi)$, $p_1(\xi)$, $\xi \in TY$, of $P_0$ and $P_1$ are the projections onto the $+i\vert \xi\vert$, respectively $-i\vert \xi \vert$ eigenspaces of the endomorphism $\gamma(\nu)\gamma(\xi)$ of $\Sigma_Y$. In other words, we have $p_0(\xi) = 1-p_1(\xi) = p_1^\perp(\xi)$ and $p_1(\xi) = 1-p_0(\xi) = p_0^\perp(\xi)$, where $p_i^\perp(\xi)$ denotes the principal symbol of $P_i^\perp$. The principal symbol of the operator $P_1^\perp P_0 + P_1 P_0^\perp$ at $\xi \in TY$ is therefore
\begin{equation*}
  p_1^\perp(\xi)p_0(\xi) + p_1(\xi)p_0^\perp(\xi) = p_0(\xi)^2 + \bigl(1-p_0(\xi)\bigr)^2 = 1.
\end{equation*}
Hence $P_1^\perp P_0 + P_1 P_0^\perp = \mathrm{id} + S$, where $S$ is a pseudodifferential operator of order $-1$, hence compact. This implies that $P_1^\perp P_0 + P_1 P_0^\perp$ is a Fredholm operator, in particular has closed range.

Now to see that $L_{X_0} + L_{X_1}$ is closed (with respect to the subspace topology induced from $L^2$), we have to show that $L_{X_0}+L_{X_1} = (L_{X_0} + L_{X_1})^c \cap C^\infty(X, \Sigma_X)$. To this end, let $\varphi \in (L_{X_0} + L_{X_1})^c$ be smooth. Since by the previous step, $(L_{X_0} + L_{X_1})^c = L_{X_0}^c + L_{X_1}^c$, we can write $\varphi = \varphi_0 + \varphi_1$ with $\varphi_i \in L_{X_i}^c$. 
Since $L_{X_0}$ and $L_{X_1}$ are not necessarily orthogonal, we do not necessarily have $P_0\varphi_1 = 0$ and $P_1\varphi_0 = 0$. However, since $P_i \varphi_i = \varphi_i$, we obtain the system of equations
\begin{equation*}
  \varphi_0 = P_0P_1\varphi_1 - P_0 \varphi , \qquad \varphi_1 = P_1P_0\varphi_0 - P_1 \varphi.
\end{equation*}
By the observations above, the principal symbols of both $P_0P_1$ and $P_1P_0$ vanish, hence they are pseudo-differential operators of order $-1$. Moreover, $P_i \varphi$ is smooth. A bootstrap argument now yields that $\varphi_0$ and $\varphi_1$ are smooth, hence $\varphi_i \in L_{X_i}$ and $\varphi \in L_{X_0}+L_{X_1}$. As the other inclusion is trivial, this finishes the proof.
\end{proof}

\begin{remark} \label{RemarkL2Version}
The proof above in fact shows that an $L^2$-version of Corollary~\ref{CorCobordism} holds. More precisely, in $L^2(Y, \Sigma_Y)$, we still have
\begin{equation*}
  L_{X_0}^c \cap L_{X_1}^c = K, \qquad \text{and} \qquad  L_{X_0}^c + L_{X_1}^c = \bigl(\gamma(\nu) K\bigr)^\perp,
\end{equation*}
where $L_{X_i}^c$ denotes the completion of $L_{X_i}$ in $L^2(Y, \Sigma_Y)$.
\end{remark}


\subsection{Construction of the Anomaly Theory} \label{SectionTwist}

In this section, we finally define the {\em anomaly theory} or {\em twist functor} 
\begin{equation*}
  T: \mathrm{Bord}_{\langle d-1, d\rangle}^{\Cl} \longrightarrow \mathrm{sAlg},
\end{equation*}
where $\mathrm{Bord}_{\langle d-1, d\rangle}^{\Cl}$ denotes the bordism category of Clifford manifolds discussed in Remark~\ref{RemarkCliffordBordism} and $\mathrm{sAlg}$ denotes the bicategory of $\bbZ_2$-graded algebras, bimodules and intertwiners. Throughout, we fix a dimension $d$ for all Clifford modules.

To start with, we define a complex vector space $W_Y$ with a real structure associated to  $(d-1)$-dimensional Clifford manifolds $Y$. Remember first that the grading operator $J$, defined in \eqref{DefinitionJ}, squares to $-1$, so that it is a complex structure when acting pointwise on $C^\infty(Y, \Sigma_Y)$. We therefore define
\begin{equation} \label{DefWY}
  W_Y \coloneqq C^\infty(Y, \Sigma_Y)_J,
\end{equation}
where the subscript $J$ indicates that we consider the space of spinors on $Y$ as a complex vector space using $J$ instead of the usual complex structure. A compatible real structure is defined by
\begin{equation*}
  \overline{\varphi} \coloneqq i \gamma(\mathbf{1})\varphi,
\end{equation*}
Clifford multiplication with the canonical vector $\mathbf{1} \in \underline{\bbR} \subset TY \oplus \underline{\bbR}$. Because Clifford multiplication is odd, this conjugation anti-commutes with $J$, hence is a real structure for $J$. A Hermitian metric on $W_Y$ is given by
\begin{equation}\label{MetricWY}
  \langle \varphi, \psi \rangle_{W_Y} \coloneqq \mathrm{Re}\langle \varphi, \psi\rangle_{L^2(Y)} + i \mathrm{Re}\langle J \varphi, \psi \rangle_{L^2(Y)};
 \end{equation}
one easily checks that $\langle \overline{\varphi}, \overline{\psi} \rangle_{W_Y} = \overline{\langle \varphi, \psi \rangle}_{W_Y}$ and that the inner product is Hermitian with respect to $J$. Notice that when passing from $Y$ to $Y^\vee$, the real structure is replaced by its negative, and when passing from $Y$ to $\overline{Y}$, $J$ is replaced by $-J$. In total, we have canonical isomorphisms
\begin{equation} \label{IdentificationsW}
  W_{Y^\vee} \cong - W_Y, \qquad W_{\overline{Y}} = \overline{W}_Y,
\end{equation}
which each are the identities on the underlying space $C^\infty(Y, \Sigma_Y)$ (when forgetting the grading and Clifford multiplication).

\medskip

For  a $d$-dimensional Clifford manifold $X$ with boundary $\partial X = Y$, remember the definitions \eqref{DefHarmonicSpinors}  of the space of harmonic spinors $\mathcal{H}_X$ on $X$ and the corresponding space $L_X \subset W_Y$ of boundary values, c.f.\ \eqref{DefinitionLagrangian}. Since $L_X$ is invariant under multiplication by $J$, it is a complex subspace of $W_Y$, and due to its invariance under multiplication by $i$, Thm.~\ref{TheoremLagrangian} shows that in $W_Y$,
\begin{equation*}
  \overline{L}_X = i \gamma(\nu)L_X = L_X^\perp, 
\end{equation*}
hence $L_X$ is a Lagrangian in $W_Y$. More generally, if $X$ is a bordism between Clifford manifolds $Y_0$ and $Y_1$, the identification \eqref{IdentificationsW} implies that we naturally have $L_X \subseteq W_{Y_0} \oplus - W_{Y_1}$. We now have the following lemma.

\begin{lemma} \label{LemmaCompositionLagrangian}
Let $X$ be a bordism from $Y_0$ to $Y_1$ and let $Z$ be a closed hypersurface that splits $X$ into two parts, $X_{0}$ and $X_{1}$. Assume that $X$ has product structure near $Y_0$, $Z$ and $Y_1$. Then the composition of the corresponding Lagrangians $L_{X_{0}}$ and $L_{X_{1}}$ is a Lagrangian, namely $L_X\subseteq W_{Y_0} \oplus - W_{Y_1}$.
\end{lemma}

\begin{figure}[h] 
\begin{center}
\includegraphics[scale=0.28, draft = false]{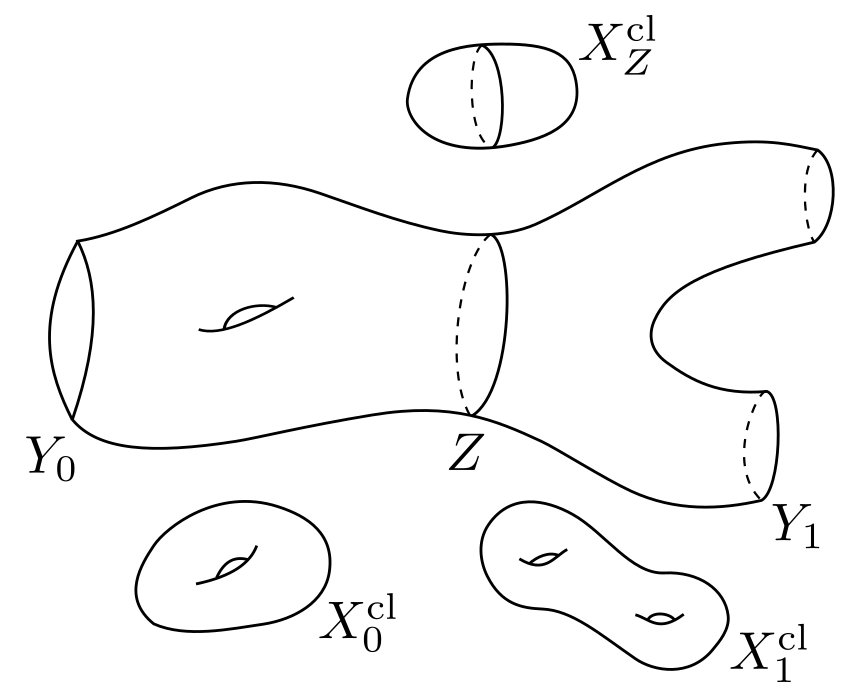}
\caption{Cutting into two pieces} \label{FigureCuttingIntoTwoPieces}
\end{center}
\end{figure}

\begin{proof}
The inclusion $L_X \subset L_{X_{1}} \circ L_{X_{0}}$ is trivial, since harmonic spinors on $X$ restrict to harmonic spinors on $X_0$ respectively $X_1$. 

Let now $(\varphi_0, \varphi_1) \in L_{X_0} \circ L_{X_1}$, where $\varphi_i \in C^\infty(Y_i, \Sigma_{Y_i})$. Then there exist harmonic spinors $\Phi_i$ on $X_i$, $i = 0, 1$, such that $\Phi_i|_{Y_i} = \varphi_i$ and such that $\Phi_1|_Z = \Phi_0|_Z \eqqcolon \psi$. Hence we obtain a continuous spinor $\Phi$ on $X$ by gluing $\Phi_0$, $\Phi_1$ together at $Z$. Then for any spinor $\Psi$ on $X$ that is compactly supported in the interior of $X$, the integration by parts formula \eqref{IntegrationByParts}, applied to $X_0$ and $X_1$ separately, yields
\begin{equation*}
\begin{aligned}
  \langle \Phi, D_X \Psi\rangle_{L^2(X)} &= \langle D_{X_1}\Phi_1, \Psi\rangle_{L^2(X_1)} +  \langle \psi, \gamma(\nu_Z)\Psi|_{Z}\rangle_{L^2(Z)}\\
  &\quad- \langle \varphi_1, \gamma(\nu_1) \Psi|_{Y_1}\rangle_{L^2(Y_1)} +\langle D_{X_0}\Phi_0, \Psi\rangle_{L^2(X_0)} \\
  &\quad+  \langle \varphi_0, \gamma(\nu_0) \Psi|_{Y_0}\rangle_{L^2(Y_0)} - \langle \psi, \gamma(\nu_Z)\Psi|_{Z}\rangle_{L^2(Z)}.
\end{aligned}
\end{equation*}
Here $\nu_i$ denotes the outward unit normal of $Y_i$ and $\nu_Z$ denotes the unit normal for $Z$ pointing towards $Y_1$. 
The first terms including $D_{X_i} \Phi_i$ are zero since $\Phi_0$, $\Phi_1$ are harmonic by assumption. 
The terms involving the boundary restrictions to $Z$ cancel because $\Phi_0$ and $\Phi_1$ match at the boundary, and the other terms vanish since $\Psi|_{Y_1} = 0$, $\Psi|_{Y_0} = 0$, as $\Psi$ is compactly supported in the interior of $X$. This shows that $\Phi$ is a weak solution to the equation $D_X\Phi = 0$ in the interior of $X$. By elliptic regularity, it is therefore a strong solution, hence smooth on all of $X$; in other words, $\Phi \in \mathcal{H}_X$. Hence $(\varphi_0, \varphi_1) \in L_X$, as claimed.
\end{proof}

\begin{remark} \label{RemarkFunktorBordLagRel}
The above observations can be used to construct a functor
\begin{equation*}
  \mathcal{L}: \mathrm{Bord}_{\langle d-1, d\rangle}^{\Cl} \longrightarrow \mathrm{LagRel},
\end{equation*}
where $\mathrm{LagRel}$ is the category of Lagrangian relations described in Remark~\ref{RemarkCategory}. Since the objects of $\mathrm{LagRel}$ are {\em complete} complex vector spaces with a real structure and a polarization, we define $\mathcal{L}(Y) \coloneqq W_Y^c = L^2(Y, \Sigma_Y)_J$, the Hilbert space completion of $W_Y$. A polarization on this Hilbert space is given as follows. Define the boundary Dirac operator operator $A_Y$ by the formula
\begin{equation} \label{BoundaryDirac}
  A_Y \varphi \coloneqq \sum_{j=1}^{d-1} \gamma(\mathbf{1})\gamma(e_j) \nabla^\Sigma_{e_j} \varphi
\end{equation}
in terms of a local orthonormal basis $e_1, \dots, e_{d-1}$ of $Y$. This is a self-adjoint, elliptic differential operator on the closed manifold $Y$. Let
\begin{equation} \label{DefinitionLY}
L_Y \coloneqq \bigoplus_{\lambda <0} \mathrm{Eig}(A_Y, \lambda) ~\subset~ C^\infty(Y, \Sigma_Y)
\end{equation}
be the subspace spanned by negative eigenvalues of $A_Y$. Since $\gamma(\mathbf{1})$ anti-commutes with $A_Y$, $\gamma(\mathbf{1}) L_Y$ is just the space spanned by the negative eigenvalues. Hence (as $\ker(A_Y)$ is finite-dimensional), the completion $L_Y^c$ is a sub-Lagrangian on $W_Y^c$, determining a polarization. 

It is well-known \cite[Proposition 14.2]{BoosBavnbek} that the orthogonal projection onto $L_Y$ is a pseudodifferential operator of order zero on $Y$, and its principal symbol $p(\xi)$, $\xi \in TY$, is the projection onto the $i\vert \xi \vert$ eigenspace of the endomorphism $\gamma(\mathbf{1})\gamma(\xi)$ (which is just the principal symbol of $A_Y$ at $\xi$). Therefore, it follows from Thm.~\ref{ThmCobordism} that if $Y$ bounds a $d$-dimensional Clifford manifold $X$, then $L_Y^c$ and $L_X^c$ are close in $W_Y^c$, in the sense of Remark~\ref{RemarkCategory}. In particular, if $X$ is a $d$-dimensional bordism from $Y_1$ to $Y_0$, then the completion $L^c_X$  is close to $L_{Y_0} \oplus \overline{L}_{Y_1}$ in $W_{Y_0 \sqcup Y_1^\vee}^c$. This implies that $\mathcal{L}(X) \coloneqq L_X^c$ is indeed a morphism in the category $\mathrm{LagRel}$ from $(W_{Y_1}^c, [L_{Y_1}])$ to $(W_{Y_0}^c, [L_{Y_0}])$.
Moreover, it follows from Thm.~\ref{ThmComposition} that the composition $L_{01}^c \circ L_{12}^c$ is a Lagrangian, a result that does not follow directly from Lemma~\ref{LemmaCompositionLagrangian}.
\end{remark}

We would now like to post-compose the assignment $Y \mapsto W_Y$ and  $X \mapsto L_X$ as discussed in Remark~\ref{RemarkFunktorBordLagRel} above with the ``second quantization functor'' described at the beginning of \S\ref{SectionBimodules}. The problem with this, however, is that this second quantization procedure is {\em not} functorial due to the {\em anomaly} given be the space $K$ in Thm.~\ref{ThmGluing}. However, notice that when passing from the space of harmonic spinors $\mathcal{H}_X$ to the Lagrangian $L_X$, we lose some information, as the map has a kernel; it is the finite-dimensional harmonic spinors on the {\em closed components} of $X$, c.f.\ Lemma~\ref{LemmaKernel}. The solution to the aforementioned problem is to take this information into account; the definition is as follows.

\begin{definition}[The Twist]
For a $(d-1)$-dimensional closed Clifford manifold $Y$, we set 
\begin{equation*}
  T(Y) \coloneqq \Cl(W_Y).
\end{equation*}
For a bordism $X$ between two Clifford manifolds $Y_0$ and $Y_1$ having product structure near the boundary, we set
\begin{equation}
  T(X) \coloneqq \Lambda L_X \otimes \Lambda^{\mathrm{top}} (\mathcal{H}_{X^{\cl}})_J,
\end{equation}
which is naturally a $T(Y_0)$-$T(Y_1)$-bimodule, as discussed above. Here $X^{\cl}$ denotes the closed part of $X$ and the subscript $J$ indicates that we consider $\mathcal{H}_{X^{\cl}}$ as a complex vector space using $J$.
\end{definition}

In the definition above, for simplicity, we suppressed the boundary identification morphisms for the bordism $X$.

\begin{remark} \label{RemIandComplexConjugate}
If we split $\mathcal{H}_X = \mathcal{H}^+_X \oplus \mathcal{H}^-_X$, according to the grading of $\Sigma_X$, then $(\mathcal{H}_X^+)_J = \mathcal{H}_X^+$ as complex vector space and $(\mathcal{H}_X^-)_J \cong \overline{\mathcal{H}_X^-}$, the complex conjugate vector space. Hence
\begin{equation*}
  \Lambda^{\mathrm{top}}(\mathcal{H}_{X^{\cl}})_J = \Lambda^{\mathrm{top}}\mathcal{H}_{X^{\cl}}^+ \otimes\overline{\Lambda^{\mathrm{top}}\mathcal{H}_{X^{\cl}}^-},
\end{equation*}
the conjugate of the {\em determinant line} of the Dirac operator $D_{X^{\cl}}$. Hence the functor $T$ indeed assigns the conjugate determinant line to closed $d$-dimensional manifolds, as desired. 
\end{remark}

\begin{remark}
  With a view on Remark~\ref{RemarkFunktorBordLagRel}, there is an $L^2$-version $T^c$ of $T$, given by $T^c(Y) = \Cl(W_Y^c)$ and $T^c(X) = \Lambda L_X^c$. Inclusion provides a natural transformation $T \Rightarrow T^c$ of functors.
\end{remark}

To discuss functoriality, we need the following lemma.

\begin{lemma} \label{LemmaAssumptions}
Let $X$ be a bordism of Clifford manifolds from $Y_0$ to $Y_1$ and let $Z$ be a closed hypersurface that splits $X$ into two parts, $X_{0}$ and $X_{1}$. Assume that $X$ has product structure near $Y_0$, $Z$ and $Y_1$. Then the corresponding Lagrangians $L_{X_{0}}$ and $L_{X_{1}}$ satisfy the assumptions $(1)$ and $(2)$ of Thm.~\ref{ThmGluing}.
\end{lemma}

\begin{proof}
Set $L \coloneqq L_{X_0} \oplus L_{X_1}$. To verify property (1) we need to show  that the image of the map 
\begin{equation*}
\sigma: L = L_{X_0} \oplus L_{X_1} \rightarrow W_Z, \qquad \sigma(\varphi_0, \varphi_1, \varphi_1^\prime, \varphi_2) = \varphi_1 - \varphi_1^\prime
\end{equation*}
is closed. To see this, suppose first that $Y_0 = Y_1 = \emptyset$. In this case, we have $\image(\sigma) = L_{X_0} + L_{X_1} = (\gamma(\nu)K)^\perp \subseteq W_{Y_1}$, by Thm.~\ref{ThmCobordism}. This is clearly closed; in fact, $L_{X_0} + L_{X_1}$ has even finite codimension in $W_{Y_1}$, since $K$ is finite-dimensional.  

If at least one of $Y_0$ and $Y_1$ is non-empty, we form the double $M = X \sqcup_{\partial X} \overline{X}$, which is then  a closed Clifford manifold. This has the hypersurface $N = Z \sqcup \overline{Z}$ that separates $M$ into a bordism $M_1 = X_1 \sqcup_{Y_1} \overline{X}_1$ from $\emptyset$ to $N$ and a bordism $M_0 = X_0 \sqcup_{Y_0}  \overline{X}_0$ from $N$ to $\emptyset$. By the previous discussion, $L_{M_0} + L_{M_1}$ has finite codimension in $W_N = W_Z \oplus W_{\overline{Z}}$. This implies that 
\begin{equation*}
P_{W_Z}(L_{M_0} + L_{M_1}) = \bigl\{ \varphi_0 + \varphi_1 \mid \exists \Phi_i \in \mathcal{H}_{M_i}: \Phi_i|_Z = \varphi_i \bigr\}
\end{equation*}
 has finite codimension in $W_Z$, where $P_{W_Z}$ is the orthogonal projection in $W_N$ onto $W_Z$ (in fact, this codimension is zero unless $X$ had closed components to begin with, since the Dirac operator on the double is invertible, c.f.\ the proof of Thm.~\ref{TheoremLagrangian}). However, since a harmonic spinor on $M_i$ restricts to a harmonic spinor on $X_i$, we have 
 \begin{equation*}
 P_{W_Z}(L_{M_0} + L_{M_1}) \subseteq P_{W_Z}(L_{X_0} + L_{X_1}) \subseteq \image(\sigma). 
 \end{equation*}
 Hence $\image(\sigma)$ has finite codimension and is therefore closed.
 
It remains to show (2), which by the observations from \S\ref{SectionLagrangians} is equivalent to showing that the adjoint $\sigma^*$ of the map $\sigma$ above has closed image. If $W_Z$ and $L$ were complete, this would follow from the observations above together with the closed range theorem; since they are not, we have to give an additional argument. To this end, form the completions $L^c_{X_i}$ of $L_{X_i}$ in $L^2(\partial X_i, \Sigma_{\partial X_i})$. Using Remark~\ref{RemarkL2Version}, the arguments above now can be easily extended to show that also the map $\sigma_c : L^{c}_{X_0} \oplus L^{c}_{X_1} \rightarrow L^2(Z, \Sigma_Z)$, obtained by extending $\sigma$ by continuity to the closure, has closed range; namely, its range contains the space $P_{W_Z}(L_{M_0}^c + L_{M_1}^c)$ of finite codimension. The closed range theorem then implies that its adjoint $\sigma_c^*: L^2(Z, \Sigma_Z) \rightarrow L^{c}_{X_0} \oplus L^{c}_{X_1}$ has closed range. 

We now claim that if $\varphi  \in \image(\sigma_c^*)$ is smooth, then in fact $\varphi \in \image(\sigma^*)$, which  finishes the proof.
To this end, notice that $\sigma^*_c(\psi) = (P_0 (0, \psi), -P_1(\psi, 0))$, where $P_i$ is the orthogonal projection onto $L_{X_i}$. With respect to the direct sum decompositions $W_{\partial X_0} = W_{Y_0} \oplus - W_Z$ and $W_{\partial X_1} = W_Z \oplus - W_{Y_1}$, write
\begin{equation*}
  P_0 ~\widehat{=}~ \begin{pmatrix} P_{00} & P_{0Z} \\ P_{Z0} & P_{ZZ} \end{pmatrix}, \qquad P_1 ~\widehat{=}~ \begin{pmatrix} P_{ZZ}^\prime & P_{1Z} \\ P_{Z1} & P_{11} \end{pmatrix}.
\end{equation*}
Now if $\sigma_c^{\ast}(\psi)$ is smooth, this implies in particular that $P_{ZZ} \psi$ and $P_{ZZ}^\prime \psi$, and also $(P_{ZZ}+P_{ZZ}^\prime)\psi$ are smooth. By Thm.~\ref{ThmCobordism}, $P_{ZZ}$ and $P_{ZZ}^\prime$ are pseudodifferential operators of order zero, and for their symbols, we have $p_{ZZ}(\xi) = 1 - p^\prime_{ZZ}(\xi)$, $\xi \in TY$. This implies that $P_{ZZ} + P_{ZZ}^\prime$ is elliptic so that $\psi$ is smooth by regularity.
\end{proof}

By virtue of the lemma above, we can now use Thm.~\ref{ThmGluing} to define isomorphisms
\begin{equation} \label{MapTau}
  \tau : T(X) \longrightarrow T(X_{0}) \otimes_{T(Z)} T(X_{1})
\end{equation}
of $T(Y_0)$-$T(Y_1)$-bimodules, in order to complete the definition of the functor $T$. These isomorphisms will be coherent in the sense that they fit into a commutative diagram of the form \eqref{CoherenceGeneral}, c.f.\ Thm.~\ref{ThmCoherence} below. To define these isomorphisms, first observe that
\begin{align*}
\mathcal{H}_{X^{\cl}} =  \mathcal{H}_{X_0^{\cl}} \oplus \mathcal{H}_{X_1^{\cl}} \oplus \mathcal{H}_{X_Z^{\cl}},
\end{align*}
where $X_i^{\cl}$ is the closed part of $X_i$, $i=0, 1$, and $X_Z^{\cl}$ is the remaining part of $X^{\cl}$, which can be identified with the connected components of $X^{\cl}$ that touch $Z$, by virtue of Lemma~\ref{LemmaKernel}, see Figure \ref{FigureCuttingIntoTwoPieces}. 
It now suffices to construct an isomorphism
\begin{equation*}
  \tau^\prime: \Lambda L_X \otimes \Lambda^{\mathrm{top}} (\mathcal{H}_{X_Z^{\cl}})_J \longrightarrow \Lambda L_{X_0} \otimes_{\Cl(W_Z)} \Lambda L_{X_1};
\end{equation*}
the isomorphism $\tau$ from \eqref{MapTau} is then just $\tau^\prime$ tensored with the canonical isomorphism
\begin{equation*}
  \Lambda^{\mathrm{top}} \bigl(\mathcal{H}_{X_0^{\cl} \sqcup X_1^{\cl}}\bigr)_J \cong \Lambda^{\mathrm{top}} \bigl(\mathcal{H}_{X_0^{\cl}} \oplus \mathcal{H}_{X_1^{\cl}}\bigr)_J \cong \Lambda^{\mathrm{top}} \bigl(\mathcal{H}_{X_0^{\cl}}\bigr)_J \otimes \Lambda^{\mathrm{top}} \bigl(\mathcal{H}_{X_1^{\cl}}\bigr)_J.
\end{equation*}
  We define $\tau^\prime$ by requiring
\begin{equation} \label{DefTauPrime}
  \tau^\prime\bigl(\Omega_X \otimes \Phi_1 \wedge \cdots \wedge \Phi_n\bigr) = \frac{\Omega_{X_0} \cdot \Phi_1|_Z \cdots \Phi_n|_Z \otimes \Omega_{X_1}}{\det(\Res_Z^*\Res_Z|_{\mathcal{H}_{X_Z}^{\cl}})},
\end{equation}
where $\Res_Z: \mathcal{H}_X \rightarrow W_Z$ is the boundary restriction map for $Z$. A few remarks are in order. First notice that by Lemma~\ref{LemmaKernel}, $\Res_Z$ (or rather its restriction to $\mathcal{H}_{X^{\cl}_Z}$) provides an isomorphism between $\mathcal{H}_{X_Z^{\cl}}$ and the space
\begin{equation*}
  K = \{ \varphi \mid (0, \varphi) \in L_{X_0}, (\varphi, 0) \in L_{X_1}\} \subset W_Z.
\end{equation*}
Hence Thm.~\ref{ThmGluing} asserts that there exists a bimodule isomorphism $\tau^\prime$ subject to \eqref{DefTauPrime} and that it is in fact uniquely determined by this requirement. 

\begin{remark}
The determinant factor  in the definition \eqref{DefTauPrime} of $\tau^\prime$ depends on the choice of a metric on $\mathcal{H}_{X_Z^{\cl}}$, respectively the determinant line $\Lambda^{\mathrm{top}}\mathcal{H}_{X_Z^{\cl}}$. The precise choice is immaterial for the construction of the functor here; however, we remark that there is a natural choice of such a metric, as the determinant line carries a natural metric, the {\em Quillen metric}, which has the property that it varies smoothly with the Riemannian metric and Clifford module structure on $X$.
\end{remark}

The following theorem now shows the coherence of these isomorphisms $\tau$, completing the construction of the twist functor $T$.

\begin{theorem} \label{ThmCoherence}
Let $X$ be a bordism between $(d-1)$-dimensional Clifford manifolds $Y_3$ and $Y_0$. Let $Y_1$, $Y_2$ be two closed hypersurfaces that split $X$ into bordisms $X_{ij}$ from $Y_j$ to $Y_i$, for $0 \leq i < j \leq 3$. Assume that $X$ has product structure near each $Y_i$. Then the diagram
\begin{equation*}
  \begin{tikzcd}[column sep={4cm,between origins}]
     & T(X) \ar[dl, bend right = 20, "\tau_{023}"'] \ar[dr, bend left = 20, "\tau_{013}"] & \\
    T(X_{02}) \otimes_{T(Y_2)} T(X_{23}) \ar[d, "\tau_{012} \otimes \mathrm{id}"'] & & T(X_{01})\otimes_{T(Y_1)} \ar[d,  "\mathrm{id} \otimes \tau_{123}"] T(X_{13}) \\
     \bigl(T(X_{01})\otimes_{T(Y_1)} T(X_{12})\bigr) \otimes_{T(Y_2)} T(X_{23}) \ar[rr, "\cong"]&  & T(X_{01})\otimes_{T(Y_1)} \bigl(T(X_{12}) \otimes_{T(Y_2)} T(X_{23})\bigr)
  \end{tikzcd}
\end{equation*}
commutes, where the bottom arrow is the associator of the tensor product.
\end{theorem}

\begin{figure}[h]
\begin{center}\includegraphics[scale=0.28, draft = false]{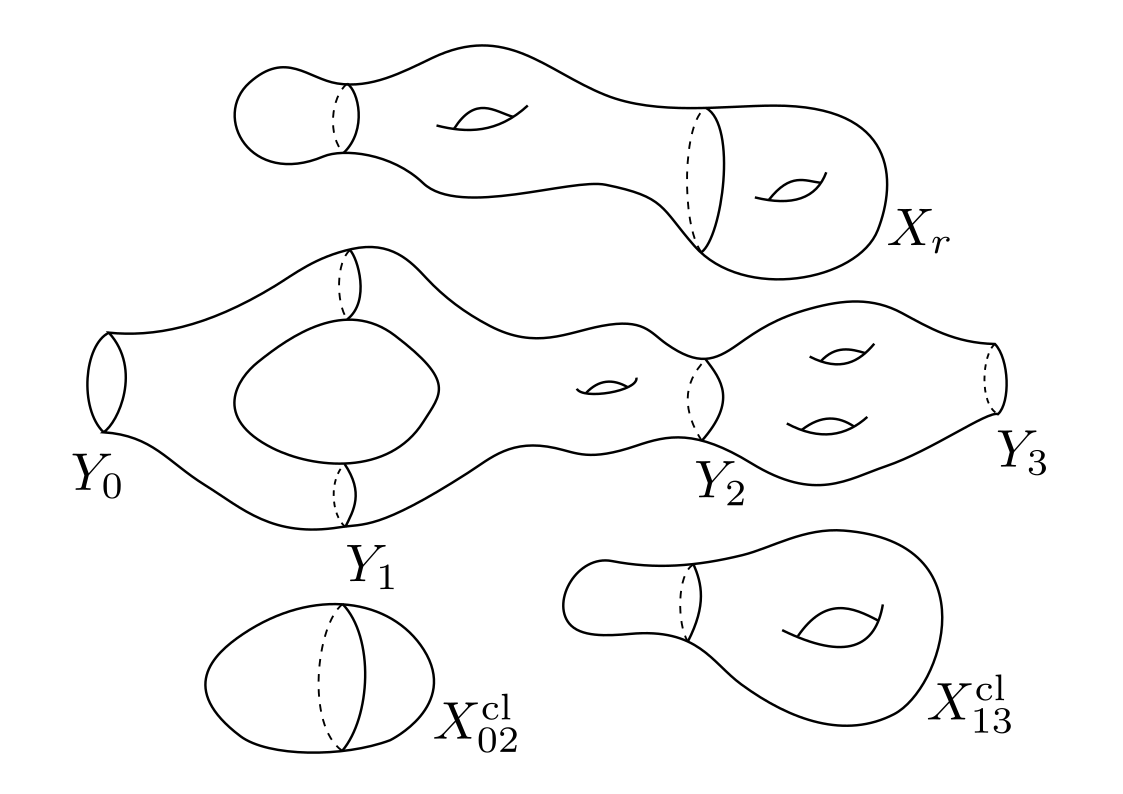}
\caption{Cutting into three pieces}\label{FigureCuttingIntoThreePieces}
\end{center}
\end{figure}

\begin{proof}
The bordism $X$ splits up into a bordism without closed components and a disjoint union of closed components.
This splitting corresponds to an exterior direct sum decomposition on the algebraic side, hence we may consider the separately the case where $X$ has no closed components and the case where $X$ is closed.

If there are no closed components, the statement follows directly from Thm.~\ref{ThmNumber} as the spaces $K_{ijk}$ are trivial in this case.

Suppose now that $X$ is closed.
The strategy of the proof is to attach commutative triangles and squares to the exterior of the commutative diagram from  to obtain a large commutative diagram containing the diagram from the theorem as outer circle.

We may assume that $X_{01}^{\cl} = X_{12}^{\cl} = X_{23}^{\cl} = \emptyset$, as the general case follows from this one after tensoring by suitable identity maps. We can then write
\begin{equation*}
  X^{\cl} = X^{\cl}_{02} \sqcup X_{13}^{\cl} \sqcup X_{{r}},
\end{equation*}
where (using our assumption) $X^{\cl}_{02}$ consists of all closed connected components of $X$ that only touch $Y_1$, $X^{\cl}_{13}$ consists of those that only touch $Y_2$ and $X_{{r}}$ consists of the rest, which are those closed connected components of $X$ that touch both $Y_1$ and $Y_2$, see Figure \ref{FigureCuttingIntoThreePieces}. Clearly, $\mathcal{H}_{X^{\cl}}$ can be split as an orthogonal direct sum in two ways,
\begin{equation*}
  \mathcal{H}_{X^{\cl}} = \mathcal{H}_{M_1^{\cl}} \oplus \mathcal{H}_{X_{13}^{\cl}} =  \mathcal{H}_{X_{02}^{\cl}} \oplus \mathcal{H}_{M_2^{\cl}}
\end{equation*}
where we wrote $M_1^{\cl} = X_{02}^{\cl} \sqcup X_r$, $M_2^{\cl} = X_{13}^{\cl} \sqcup X_r$.
Let us abbreviate $L_{ij} = L_{X_{ij}}$ for $0 \leq i < j \leq 3$ and define $K_{ijk}$ as in \eqref{DefinitionKandM}. Let $\Res_i: \mathcal{H}_X \rightarrow W_{Y_i}$ be the restriction maps, $i=1, 2$. We then have the following commutative diagram
\begin{equation*}
\begin{tikzcd}[column sep={3cm,between origins}]
\mathcal{H}_{M_1^{\cl}} \oplus \mathcal{H}_{X_{13}^{\cl}} \ar[d, "\Res_1 \oplus \Res_2"]
& 
\mathcal{H}_X^{\cl} \ar[l, equal] \ar[r, equal]
& \mathcal{H}_{M_2^{\cl}} \oplus \mathcal{H}_{X_{02}^{\cl}}  \ar[d, "\Res_1 \oplus \Res_2"]\\
K_{013} \oplus K_{123} \ar[rr, "\rho"'] & & K_{012}\oplus K_{023}, 
 \end{tikzcd}
\end{equation*}
where $\rho$ is the ``development map'' defined in \eqref{DefinitionIsoPhi}. This leads to the commutative diagram
\begin{equation} \label{LambdaTopDiagram}
\begin{tikzcd}[column sep={3cm,between origins}]
 & \Lambda^{\mathrm{top}}\mathcal{H}_X^{\cl} 
 	\ar[dl, "\frac{\Lambda^{\mathrm{top}} \Res_1}{\det(\Res_1^*\Res_1|_{M_1^{\cl}})} \otimes \frac{\Lambda^{\mathrm{top}} \Res_2}{\det(\Res_2^*\Res_2|_{X_{13}^{\cl}})}"', bend right=10] 
	\ar[dr, "\frac{\Lambda^{\mathrm{top}} \Res_1}{\det(\Res_1^*\Res_1|_{X_{02}^{\cl}})} \otimes \frac{\Lambda^{\mathrm{top}} \Res_2}{\det(\Res_2^*\Res_2|_{M_2^{\cl}})}", bend left=10]& \\
\Lambda^{\mathrm{top}}K_{013} \otimes \Lambda^{\mathrm{top}} K_{123} \ar[rr, "\frac{\Lambda^{\mathrm{top}}\rho}{\det(\rho^*\rho)}"'] & & \Lambda^{\mathrm{top}}K_{012}\otimes \Lambda^{\mathrm{top}} K_{023}, 
 \end{tikzcd}
\end{equation}
using multiplicativity of the determinant and functoriality of taking the top exterior power. Since the only difference between the isomorphisms $\tau$ defined above and the isomorphisms $\alpha$ used in \S\ref{SectionBimodules} is the determinant factor, also the diagram
\begin{equation*}
\begin{tikzcd}[column sep={7cm,between origins}]
T(X)
	\ar[r, "\mathrm{id} \otimes \frac{\Lambda^{\mathrm{top}} \Res_1}{\det(\Res_1^*\Res_1|_{M_1^{\cl}})} \otimes \frac{\Lambda^{\mathrm{top}} \Res_2}{\det(\Res_2^*\Res_2|_{X_{13}^{\cl}})}", bend left=10] 
	\ar[d, "\tau_{013} \otimes \mathrm{id}"]
	& 
	\bigl(\Lambda L_{03} \otimes \Lambda^{\mathrm{top}}K_{013}\bigr) \otimes \Lambda^{\mathrm{top}} K_{123} 
		\ar[d, "\alpha_{013} \otimes \mathrm{id}"] 
		\\
T(X_{02}) \otimes_{T(Y_2)} T(X_{23})
		\ar[r, "\mathrm{id} \otimes \mathrm{id} \otimes \frac{\Lambda^{\mathrm{top}} \Res_2}{\det(\Res_2^*\Res_2|_{X_{13}^{\cl}})}~~~", bend left=10] 
		\ar[dddr, "\mathrm{id} \otimes \tau_{123}"', bend right=30]
		&
			\bigl(\Lambda L_{01} \otimes_{\Cl(W_1)} \Lambda L_{13}\bigr) \otimes \Lambda^{\mathrm{top}} K_{123}
				\ar[d, "\cong"]
				\\
	        &
			\Lambda L_{01} \otimes_{\Cl(W_1)} \bigl(\Lambda L_{13} \otimes \Lambda^{\mathrm{top}} K_{123}\bigr)
				\ar[d, "\mathrm{id} \otimes \alpha_{123}"]
				\\
		&
			\Lambda L_{01} \otimes_{\Cl(W_1)} \bigl(\Lambda L_{12} \otimes_{\Cl(W_2)} \Lambda L_{23}\bigr)
				\ar[d, equal]
				\\
		&  T(X_{01})\otimes_{T(Y_1)} \bigl(T(X_{12}) \otimes_{T(Y_2)} T(X_{23})\bigr)
\end{tikzcd}
\end{equation*}
commutes. This connects the left half of the diagram from Thm.~\ref{ThmNumber} to the left half of the diagram of this theorem. A similar commutative diagram can be drawn for the respective right halves. Now tensoring the diagram \eqref{LambdaTopDiagram} with $\Lambda L_{03}$, the resulting diagram is used to join these two halves yielding a large commutative diagram that interpolates between the diagram from this theorem and the one from Thm.~\ref{ThmNumber}. This finishes the proof.
\end{proof}

\section{Outlook}\label{SectionOutlook}

In this expository section, we give an overview on questions that were not addressed in this paper and that remain open. In particular, in the first few paragraphs below, we sketch how to construct the free fermion itself, and how the anomaly theory $T$ can be extended above. Interestingly, it turns out that essentially all objects of interest related to the Dirac operator $-$ its eta-invariant, its determinant, its index and the $\hat{A}$-genus appear in this description. In this sense, the free fermion alone encompasses the whole index theory of the Dirac operator.

\subsection{The Free Fermion}

The purpose of this paper was not to construct the free fermion itself, but the corresponding {\em anomaly theory}, its {\em twist} in functorial field theory language. Having completed this task, the free fermion will then be a natural transformation $F: \mathbf{1} \rightarrow T$, 
\begin{equation*}
\begin{tikzcd}[column sep = {1.7cm,between origins},row sep={0.5cm,between origins}]
 & \,\ar[dd, Rightarrow, "F"] & \\
\mathrm{Bord}_{\langle d-1, d\rangle}^{\Cl}~~ \ar[rr, bend left=30, "\mathbf{1}"] \ar[rr, bend right=30, "T"'] & & \mathrm{sAlg}, \\
 & \, &
\end{tikzcd}
\end{equation*}
where $\mathbf{1}$ is the trivial field theory, which assigns $\bbC$, respectively identities.
In this section, we give a quick overview how to define it. 

To begin with, let us unravel what such a natural transformation consists of. First, for every object of $\mathrm{Bord}_{\langle d-1, d\rangle}^{\Cl}$, in other words, a $(d-1)$-dimensional closed Clifford manifold, $F$ assigns a morphism $F(Y): T(Y) \rightarrow \mathbf{1}(Y) = \bbC$ in the category $\mathrm{sAlg}$, in other words, $F(Y)$ will be a $\Cl(W_Y)$-module. Now if $Y_0$, $Y_1$ are two such manifolds and $X$ is a bordism between the two, then the naturality of the transformation in this higher categorical context is the additional data of a homomorphism
\begin{equation*}
  F(X): T(X) \otimes_{T(Y_1)} F(Y_1)\longrightarrow F(Y_0)   
\end{equation*}
satisfying certain coherence conditions related to the composition of bordisms.

\medskip

To define the $\Cl(W_Y)$-module $F(Y)$ for a $(d-1)$-dimensional closed Clifford manifold $Y$, one takes it of the form $F(Y) = \Lambda L$ for some Lagrangian $L$ in $W_Y$. Here $L$ should be taken from the space\footnote{Here one is flexible to modify the notion of closeness for subspaces $L$ of $W_Y$. One should at least assume that the difference $P_L - P_{L_Y}$ is a Hilbert-Schmidt operator; however, in this case, it may even be suitable to require $P_L - P_{L^\prime}$ being a smoothing operator.}
\begin{equation} \label{SpaceOfLagrangians}
  \mathfrak{Lag}_Y = \bigl\{ L \subset W_Y ~\bigl|~ L~ \text{Lagrangian close to}~L_Y\bigr\},
\end{equation}
where $L_Y$ is the sub-Lagrangian defined in \eqref{DefinitionLY}.
However, as $L_Y$ is not always a Lagrangian due to the possible presence of harmonic spinors, there is no canonical choice for such a Lagrangian $L$. 
A possible approach to define $F$ is now to just choose a Lagrangian $L$ for each such manifold $Y$ (using the axiom of choice), at the expense of $F$ depending on this choice. 

To provide a glimpse into a further understanding of this problem, we refer to Thm.~1.29 of \cite{Prat}, which states that if $L_1$ and $L_2$ are two such Lagrangians that are close in the Hilbert-Schmidt sense, the space of $\Cl(W_Y)$-linear homomorphisms from $(\Lambda L_1)^c$ to $(\Lambda L_2)^c$ is a complex line\footnote{This result is only true when passing to the completions of $\Lambda L_i$, hence it seems suitable to do this throughout at this point.}. This gives rise to a gerbe on the moduli space $\mathscr{M}_Y^{\Cl}$ of Clifford structures on $Y$, and the choice needed to define the desired functor $F$ on objects is precisely the choice of a trivialization of this gerbe. In particular, in order for  $F$ to have the property that the modules $F(Y)$ depend smoothly on the Clifford structure on $Y$, one needs this gerbe to admit a trivialization in the first place. We will not further elaborate on this problem here.

\medskip

Let now $X$ be a closed $d$-dimensional Riemannian manifold with $d$-dimensional Clifford module, seen as an automorphism of the empty manifold in  $\mathrm{Bord}_{\langle d-1, d\rangle}^{\Cl}$. In this case, $F(X)$ is a homomorphism
\begin{equation*}
  F(X): T(X) = \Lambda^{\mathrm{top}} \mathcal{H}^+(X) \otimes \overline{\Lambda^{\mathrm{top}}\mathcal{H}^-(X)} \longrightarrow \bbC  ,
\end{equation*}
c.f.\ Remark~\ref{RemIandComplexConjugate}. In other words, $F(X)$ is nothing but an element of the determinant line of $X$. There is a canonical such element: The determinant $\det_\zeta(D_X)$ of the Dirac operator of $X$, the definition of which is quite tautological in this description of the determinant line: It is one if $\mathcal{H}(X) = \{0\}$ and zero otherwise. It is then a non-trivial statement due to Quillen that the determinant lines glue together to a line bundle on the moduli space $\mathscr{M}_X^{\Cl}$ of Clifford structures on $X$ and that the determinant is a smooth section of this line bundle; for details, see e.g.\ \cite{BGV}.

It is now a very interesting (and possibly quite challenging) task to extend this definition to manifolds $X$ with boundary. Since each Lagrangian $L \in \mathfrak{Lag}_Y$ provides an elliptic boundary condition, giving rise to a self-adjoint Dirac operator $D_{X, L}$, and a regularized determinant $\det_\zeta(D_{X, L})$, we expect these determinants to be involved in the construction of the homomorphisms $F(X)$. Similar to the arguments in \cite{Kandel}, verifying functoriality will then use gluing formulas for the determinant of the Dirac operator (see e.g.~\cite{HuangLee}). We believe that these (somewhat complicated) gluing formulas will take a very natural form in this setting.

\subsection{Von Neumann Algebras and Connes Fusion}

In this paper, we constructed the anomaly theory with values in the category $\mathrm{sAlg}$ of ($\bbZ_2$-graded) algebras, bimodules on intertwiners. In other words, we worked in the purely algebraic setting, where no topology or norm was put on algebras or modules. It is another challenging task to enhance this construction to a functor
\begin{equation*}
  T : \mathrm{Bord}_{\langle d-1, d\rangle}^{\Cl} \longrightarrow \mathrm{svN},
\end{equation*}
where $\mathrm{svN}$ denotes the bicategory of $\bbZ_2$-graded von Neumann algebras, Hilbert space bimodules and intertwiners described in \S4 of \cite{StolzTeichnerElliptic}. The composition of bimodules in this bicategory is given by the fusion product of Connes for Hilbert space bimodules.

\medskip

To begin with, let $X$ be a bordism from $Y_1$ to $Y_0$ and let 
\begin{equation*}
\mathcal{F}_X \coloneqq (\Lambda L_X)^c
\end{equation*}
 be the Hilbert space completion of the $\Cl(W_{Y_0})$-$\Cl(W_{Y_1})$-bimodule $\Lambda L_X$. The Clifford algebras $\Cl(W_i)$ act on $\Lambda L_X$ via bounded operators, hence we obtain $\Cl(W_{Y_0}) \subseteq B(\mathcal{F}_X)$ and $\Cl(W_{Y_1})^{\mathrm{op}} \subset \Cl(W_{Y_0})^\prime \subseteq B(\mathcal{F}_X)$, the commutant. Taking the bicommutant of these subalgebras (or, equivalently, the weak closure) in $B(\mathcal{F}_X)$ then give rise to von Neumann algebras $\Cl(W_0)^{\prime\prime}$ and $\Cl(W_1)^{\prime\prime}$, turning $\mathcal{F}_X$ into a Hilbert space $\Cl(W_0)^{\prime\prime}$-$\Cl(W_1)^{\prime\prime}$-bimodule.
It turns out that these von Neumann algebras are always factors of type $I_\infty$ if $X$ is connected\footnote{There was a false claim in earlier versions of this paper (including the published version) that also other factors can appear.}.
If $X$ is not connected, $\Cl(W_0)^{\prime\prime}$ is a tensor product of the factors corresponding to the individual components.

\medskip


To establish functoriality for bordisms $X_0$, $X_1$ with a common boundary $Z$, one needs to find a canonical isomorphism
\begin{equation*}
\tau: \mathcal{F}_X \longrightarrow \mathcal{F}_{X_0} \boxtimes_{\Cl(W_Z)^{\prime\prime}}\mathcal{F}_{X_1}
\end{equation*}
 of Hilbert $\Cl(W_{Y_0})^{\prime\prime}$-$\Cl(W_{Y_1})^{\prime\prime}$-bimodules, where $X = X_0 \sqcup_Z X_1$ and $\boxtimes$ denotes the Connes fusion product. 
 The problem here is that while naturally $\Lambda L_X \subset \mathcal{F}_X$, the Connes fusion product $\mathcal{F}_{X_0} \boxtimes_{\Cl(W_Z)^{\prime\prime}}\mathcal{F}_{X_1}$  does {\em not} canonically contain the algebraic tensor product of $\Lambda L_{X_0} \otimes_{\Cl(W_Z)} \Lambda L_{X_1}$, which makes it hard to put the Gluing Thm.~\ref{ThmGluing} to use in this context. The solution to this problem is a question of ongoing research.

\subsection{Higher Bordism Categories}

Lurking behind various corners of this paper is the concept of {\em extended functorial field theory}. At its heart lies the observation that manifolds cannot only be cut in one direction to obtain bordisms with closed boundaries, but in several directions, giving rise to bordisms of manifolds with boundaries and corners. Consequently, at least heuristically, one can form the bordism category
\begin{equation*}
  \mathrm{Bord}_{\langle d-k, \dots, d\rangle}^{\Cl},
\end{equation*}
which is in fact a higher category, a $k$-category: objects are $(d-k)$-dimensional Clifford manifolds with corners of codimension up to $k$; morphisms are $(d-k+1)$-dimensional Clifford bordisms between such manifolds; 2-morphisms are $(d-k+2)$-dimensional bordisms between bordisms and so forth. Just as the bordism 1-category (Remark~\ref{RemarkCliffordBordism}), it is not an easy matter making this concept rigorous, c.f.\ e.g.\ \cite{StolzTeichnerElliptic, StolzTeichnerField, AyalaGeometriCob}. 

While the formal structure of these bordism categories may be rather aloof, the analytic situation is rather transparent: One encounters questions about geometric invariants on manifolds with corners, together with their gluing properties. The concept of higher bordism categories can then be used as a guiding principle to find the correct questions. It is this spirit that we will display in the next sections.

\subsection{Even vs.\ odd Clifford modules}

So far, we only discussed {\em even} Clifford manifolds. An {\em odd} Clifford manifold is essentially the same, except that the Clifford module is not required to be graded. This gives rise to {\em odd Clifford bordism categories} $\mathrm{Bord}_{\langle d-1, d, d+1\rangle}^{\Cl, \mathrm{odd}}$.

The relation to the even version is as follows. Remember that a $d$-dimensional Riemannian manifold $X$ with a $d$-dimensional even Clifford module induces a $d$-dimensional even Clifford module $\Sigma_{\partial X}$ onto its boundary $\partial X$ by restriction. It also induces a $(d-1)$-dimensional odd Clifford module on $\partial X$, by setting 
\begin{equation*}
{\Sigma}^{\mathrm{odd}}_{\partial X} \coloneqq \Sigma_X^+, \qquad {\gamma}^{\mathrm{odd}}_{\partial X}(w) = \gamma_X(\nu)\gamma_X(w),
\end{equation*}
where $\nu$ is the outward pointing normal vector. One could also define ${\Sigma}^{\mathrm{odd}}_{\partial X} \coloneqq \Sigma_X^-$ instead; multiplication by $\gamma(\nu)$ gives an isomorphism between the two definitions. Conversely, for a $(d-1)$-dimensional odd Clifford module $({\Sigma}_Y^{\mathrm{odd}}, {\gamma}^{\mathrm{odd}})$ on a $(d-1)$-dimensional manifold $Y$, setting
\begin{equation*}
  \Sigma_Y \coloneqq {\Sigma_Y}^{\mathrm{odd}} \oplus {{\Sigma}}^{\mathrm{odd}}_Y, \qquad 
  \gamma(w) = \begin{pmatrix} \gamma^{\mathrm{odd}}(w) & 0 \\ 0 & {\gamma}^{\mathrm{odd}}(w) \end{pmatrix}, \quad
  \gamma(\mathbf{1}) = \begin{pmatrix} 0 & -\Id \\ \Id& 0 \end{pmatrix},
\end{equation*}
for $w \in TY$ provides a $d$-dimensional even Clifford modules $(\Sigma_Y, \gamma)$ on $Y$.

Similarly, if $X$ is a $(d+1)$-dimensional manifold with $(d+1)$-dimensional odd Clifford module, it induces not only the structure of a $(d+1)$-dimensional odd Clifford module on its boundary $\partial X$, but also that of a $d$-dimensional even Clifford module, by just defining the grading operator by $J\coloneqq \gamma(\nu)$, Clifford multiplication with the exterior normal vector. Conversely, a graded $\Cl_d$-module can be turned into an ungraded $\Cl_{d+1}$-module, where the Clifford multiplication by the extra basis vector is given via the grading operator $J$.

In general, on $(d-k)$-dimensional Clifford manifolds, there is a similar correspondence between $d$-dimensional even Clifford modules and $(d-k)$-dimensional Clifford modules of parity $k ~ (\mathrm{mod} ~2)$.

\begin{remark}
It is not hard to phrase the results of this paper in these terms, always considering $(d-1)$-dimensional odd Clifford modules on $(d-1)$-dimensional manifolds instead of $d$-dimensional even ones. For example, with a view on Remark~\ref{CARalgebra}, the CAR-algebra over the complex vector space $C^\infty(Y, \Sigma^{\mathrm{odd}}_Y)$ is canonically isomorphic to the Clifford algebra over the complex vector space $C^\infty(Y, \Sigma_Y)_J$ with real structure given by $\gamma(\mathbf{1})$. However, since a large part of our discussion depended on restricting spinors to the boundary, the current setup seemed more natural.
\end{remark}

\subsection{Extension above: The Dai-Freed theory} \label{SectionFreedDaiTheory}

It is generally believed that the anomaly theory $T$ extends to a functor
\begin{equation*}
 T: \mathrm{Bord}_{\langle d-1, d, d+1\rangle}^{\Cl, \mathrm{odd}} \longrightarrow \mathrm{sAlg}.
\end{equation*}
The top part of this extended theory, which is a functor
\begin{equation} \label{FreedDaiTheory}
  T: \mathrm{Bord}^{\Cl, \mathrm{odd}}_{\langle d, d+1\rangle} \longrightarrow \mathrm{sVect}
\end{equation}
is well-known; it was described by Dai and Freed over 20 years ago \citep{FreedDai}. We already know what this theory assigns to objects, which are closed $d$-dimensional manifolds $X$ together with a $(d+1)$-dimensional ungraded Clifford module $\Sigma_X$ or, what is the same thing, a $d$-dimensional graded Clifford module: It is the conjugate determinant line
\begin{equation*}
  T(X) = \Lambda^{\mathrm{top}} \mathcal{H}^+(X) \otimes \overline{\Lambda^{\mathrm{top}} \mathcal{H}^-(X)} = \overline{\mathrm{Det}}_X.
\end{equation*}
The main observation of Dai and Freed then was that if $M$ is a $(d+1)$-dimensional (ungraded) Clifford manifold with boundary $\partial M = X$, its {\em exponentiated $\eta$-invariant} 
\begin{equation} \label{ExpEtaInvariant}
\tau_M = e^{2\pi i \xi_M}, \qquad \text{where} \qquad \xi_M = \frac{\eta_M + \dim \ker D_M}{2},
\end{equation}
can be canonically understood as an element of $\overline{\mathrm{Det}}_X$; here $\eta_M$ denotes the $\eta$-invariant of Atiyah, Patodi and Singer \cite{APS}. 

Namely, they observe that $\overline{\mathrm{Det}}_X$ can be naturally identified with the space of functions on the space $\mathfrak{Lag}_X$ of Lagrangians in $C^\infty(X, \Sigma_X)$ that satisfy a certain transformation rule. More precisely, while there is no real structure in this context, a Lagrangian is just defined to be a subspace $L \subset W_X$ (where $W_X \coloneqq C^\infty(X, \Sigma_X)$) such that $L^\perp = JL$. Such a Lagrangian is always the graph of a unitary transformation $U: W_X^+ \rightarrow W_X^-$, and conversely,  the graph of any such transformation is a Lagrangian. The transformation rule  for functions $f: \mathfrak{Lag}_X \rightarrow \bbC$ alluded to before is now
\begin{equation} \label{TrafoRule}
f(L_2) = \det(U_1^{-1}U_2) f(L_1),
\end{equation}
 where $L_i = \mathrm{Graph}(U_i)$. Notice that assuming both $L_1$ and $L_2$ to be Hilbert-Schmidt-close to the Lagrangian $L_X$ considered in \eqref{DefinitionLY}, we obtain that the difference $U_1^{-1}U_2 - \mathrm{id}$ is trace-class, hence the Fredholm determinant $\det(U_1^{-1}U_2)$ is well-defined. Since a function satisfying \eqref{TrafoRule} is determined by its value at a single $L \in \mathfrak{Lag}_X$, the set of such functions forms a complex line, which is canonically isomorphic to $\overline{\mathrm{Det}}_X$ by Prop.~2.15 of \cite{FreedDai}.

Now on the other hand, any $L \in \mathfrak{Lag}_X$ is a self-adjoint elliptic boundary condition for $D_M$, giving rise to an eta-invariant $\eta_{M}(L)$ and an integer $\dim \ker D_{M, L}$. Dai and Freed then verify that indeed, the function $\tau_M: \mathfrak{Lag}_X \rightarrow \bbC$ defined using \eqref{ExpEtaInvariant} satisfies the transformation rule \eqref{TrafoRule}  and consequently defines an element of the determinant line (c.f.\ \cite[Thm.~1.4]{FreedDai}).

In particular, if $M$ is a bordism from $X_1$ to $X_0$, one naturally has 
\begin{equation*}
\overline{\mathrm{Det}}_{\partial M} = \overline{\mathrm{Det}}_{X_0}\otimes {\mathrm{Det}}_{X_1} = \mathrm{Hom}\bigl(\overline{\mathrm{Det}}_{X_1}, \overline{\mathrm{Det}}_{X_0}\bigr),
\end{equation*}
so $T(M) \coloneqq \tau_M$ can be naturally seen as a homomorphism $T(M): T(X_1) \rightarrow T(X_0)$. The {\em gluing formula}, \cite[Prop.~4.5]{FreedDai}, then states that this assignment behaves functorial with respect to gluing of bordisms. 

It is an open question how to ``glue together'' the Dai-Freed-functor with the one constructed in this paper; this would require the theory of eta-invariants for manifolds with corners of codimension two.

\subsection{A second extension: Index theory} \label{SectionSecondExtension}

Conjecturally, one can extend the Dai-Freed theory \eqref{FreedDaiTheory} even further, to a functor defined on the bordism bicategory $\mathrm{Bord}_{\langle d, d+1, d+2\rangle}^{\Cl}$. To describe the target of this extended theory, notice that the theory described in \S\ref{SectionFreedDaiTheory} in fact takes values in the subcategory $\mathrm{sLine}^\times \subset \mathrm{sVect}$, the category of graded lines, together with invertible homomorphisms. The extended theory will then be a functor
\begin{equation*}
  T: \mathrm{Bord}_{\langle d, d+1, d+2\rangle}^{\Cl} \longrightarrow \mathrm{sLine}^\times_{\bbZ},
\end{equation*}
where $\mathrm{sLine}_{\bbZ}^\times$ is a bicategory, obtained from $\mathrm{sLine}$ by adding 2-morphisms, as follows. Let $L_1, L_2$ be two complex lines. Since $\mathrm{Hom}(L_1, L_2)^\times$ is a $\bbC^\times$-torsor (i.e.\ a set with a free and transitive action of $\bbC^\times$, given by multiplication), its fundamental group is {\em canonically} isomorphic to $\bbZ$, after choosing a generator of $\pi_1(\bbC^\times)$, once and for all. Consequently, the universal cover $(\mathrm{Hom}(L_1, L_2)^\times)^{\sim}$ has a canonical $\bbZ$-action. This means that an invertible homomorphism $z \in \mathrm{Hom}(L_1, L_2)^\times$ determines a $\bbZ$-torsor 
\begin{equation*}
S_z \coloneqq \pi^{-1}(z) \subset (\mathrm{Hom}(L_1, L_2)^\times)^{\sim}.
\end{equation*}
Given invertible homomorphisms $z, z^\prime \in \mathrm{Hom}(L_1, L_2)^\times$, we can therefore consider isomorphisms of $\bbZ$-torsors $\varphi: S_z \rightarrow S_{z^\prime}$; adding these as 2-morphisms gives the bicategory $\mathrm{sLine}_{\bbZ}$.

\medskip

The functor $T$ now assigns to a closed $d$-dimensional manifold $X$ the conjugate determinant line $T(X) = \overline{\Det}_X$, and to a $(d+1)$-dimensional bordism $M$ from $X_1$ to $X_0$ the $\bbZ$-torsor $T(M) = S_{\tau_M}$ determined by the element $\tau_M \in \Hom(\overline{\Det}_{X_1}, \overline{\Det}_{X_0})$, c.f.\ \S\ref{SectionFreedDaiTheory} (since homomorphisms $z \in \Hom(L_1, L_2)$ are in 1:1 correspondence with the corresponding $\bbZ$-torsors $S_z$, this really extends the theory described above). 

In particular, if $M$ is closed so that $X_0 = X_1 = \emptyset$, we have $\tau_M \in \bbC^\times$. Under the identification $(\bbC^\times)^{\sim} = \bbC$ where the projection to $\bbC^\times$ is just the exponential map, we can identify
\begin{equation*}
  T(M) = S_{\tau_M} = \frac{\eta_M + \dim \ker D_M}{2} + \bbZ,
\end{equation*}
compare \eqref{ExpEtaInvariant}. Notice that while $\eta_M + \dim \ker D_M$ is not continuous under a change of metric, the jump size is in $2\bbZ$, hence as a {\em set}, $S_{\tau_M}$ depends continuously on the geometric data.

To describe the functor $T$, let first $W$ be a $(d+2)$-dimensional Clifford bordism between {\em closed} Clifford manifolds $M_1$ and $M_0$, which we assume to have product structure near the boundary. In this case, the Atiyah-Patodi-Singer index theorem \cite{APS} states that
\begin{equation*}
  \mathrm{ind}(D_{W, \mathrm{APS}}) = \int_W \hat{A}(TW) + \frac{\dim\ker(A_{M_0}) + \dim\ker(A_{M_1})}{2} + \frac{\eta_{M_0} - \eta_{M_1}}{2},
\end{equation*}
where $D_{W, \mathrm{APS}}$ denotes the Dirac operator on $W$ with Atiyah-Patodi-Singer boundary conditions, and $A_{M_i}$ is the induced boundary Dirac operator, as in \eqref{BoundaryDirac}. In particular, this implies that if we set
\begin{equation*}
  T(W) := \int_W \hat{A}(TW),
\end{equation*}
we have the equality (of sets) $T(M_0) + T(W) = T(M_1)$, in other words, we get a map of $\bbZ$-torsors $T(W): T(M_1) \rightarrow T(M_0)$. In particular, if $W$ has no boundary, then $T(W) = \mathrm{ind}(D_W) \in \bbZ$ is an automorphism of the trivial $\bbZ$-torsor $\bbZ \subset \bbC$. Functoriality in this case is trivial, since $T$ is locally determined.

Now if $M_0$, $M_1$ themselves have a boundary, so that $W$ has corners of codimension two, one should be able to use Thm.~1.9 of \cite{FreedDai} on the variation of the determinant-line-valued eta invariant to generalize this result to this setting. Compare also the paper \cite{HassellMazzeoMelrose}, where an index formula for manifolds with corners of codimension two for the signature operator is proved.

\nocite{*}
\bibliography{literature.bib}

\begin{thebibliography}{10}

\bibitem{AtiyahTFT}
M.~Atiyah.
\newblock Topological quantum field theories.
\newblock {\em Inst. Hautes \'{E}tudes Sci. Publ. Math.}, 68:175--186 (1989),
  1988.

\bibitem{APS}
M.~F. Atiyah, V.~K. Patodi, and I.~M. Singer.
\newblock Spectral asymmetry and {R}iemannian geometry. {I}.
\newblock {\em Math. Proc. Cambridge Philos. Soc.}, 77:43--69, 1975.

\bibitem{AyalaGeometriCob}
D.~Ayala.
\newblock {\em Geometric cobordism categories}.
\newblock ProQuest LLC, Ann Arbor, MI, 2009.
\newblock Thesis (Ph.D.)--Stanford University.

\bibitem{BaerNodal}
C.~B\"{a}r.
\newblock On nodal sets for {D}irac and {L}aplace operators.
\newblock {\em Comm. Math. Phys.}, 188(3):709--721, 1997.

\bibitem{BGV}
N.~Berline, E.~Getzler, and M.~Vergne.
\newblock {\em Heat kernels and {D}irac operators}, volume 298 of {\em
  Grundlehren der Mathematischen Wissenschaften [Fundamental Principles of
  Mathematical Sciences]}.
\newblock Springer-Verlag, Berlin, 1992.

\bibitem{BoosBavnbekLesch}
B.~Boo\ss-Bavnbek and M.~Lesch.
\newblock The invertible double of elliptic operators.
\newblock {\em Lett. Math. Phys.}, 87(1-2):19--46, 2009.

\bibitem{BoosBavnbek}
B.~Boo\ss-Bavnbek and K.~P. Wojciechowski.
\newblock {\em Elliptic boundary problems for {D}irac operators}.
\newblock Mathematics: Theory \& Applications. Birkh\"{a}user Boston, Inc.,
  Boston, MA, 1993.

\bibitem{FreedDai}
X.~Dai and D.~S. Freed.
\newblock {$\eta$}-invariants and determinant lines.
\newblock {\em J. Math. Phys.}, 35(10):5155--5194, 1994.
\newblock Topology and physics.

\bibitem{FreedHopkins}
D.~S. {Freed} and M.~J. {Hopkins}.
\newblock {Reflection positivity and invertible topological phases}.
\newblock {\em ArXiv e-prints}, Apr. 2016.

\bibitem{HassellMazzeoMelrose}
A.~Hassell, R.~Mazzeo, and R.~B. Melrose.
\newblock A signature formula for manifolds with corners of codimension two.
\newblock {\em Topology}, 36(5):1055--1075, 1997.

\bibitem{HuangLee}
R.-T. Huang and Y.~Lee.
\newblock The gluing formula of the zeta-determinants of {D}irac {L}aplacians
  for certain boundary conditions.
\newblock {\em Illinois J. Math.}, 58(2):537--560, 2014.

\bibitem{Kandel}
S.~Kandel.
\newblock Functorial quantum field theory in the {R}iemannian setting.
\newblock {\em Adv. Theor. Math. Phys.}, 20(6):1443--1471, 2016.

\bibitem{Kriz}
I.~Kriz.
\newblock On spin and modularity in conformal field theory.
\newblock {\em Ann. Sci. \'{E}cole Norm. Sup. (4)}, 36(1):57--112, 2003.

\bibitem{MickelssonScott}
J.~Mickelsson and S.~Scott.
\newblock Functorial {QFT}, gauge anomalies and the {D}irac determinant bundle.
\newblock {\em Comm. Math. Phys.}, 219(3):567--605, 2001.

\bibitem{MuellerSzabo}
L.~M\"uller and R.~J. Szabo.
\newblock Extended {Q}uantum {F}ield {T}heory, {I}ndex {T}heory, and the
  {P}arity {A}nomaly.
\newblock {\em Comm. Math. Phys.}, 362(3):1049--1109, 2018.

\bibitem{Palais}
R.~S. Palais.
\newblock {\em Seminar on the {A}tiyah-{S}inger index theorem}.
\newblock With contributions by M. F. Atiyah, A. Borel, E. E. Floyd, R. T.
  Seeley, W. Shih and R. Solovay. Annals of Mathematics Studies, No. 57.
  Princeton University Press, Princeton, N.J., 1965.

\bibitem{PontoShulman}
K.~Ponto and M.~Shulman.
\newblock Traces in symmetric monoidal categories.
\newblock {\em Expo. Math.}, 32(3):248--273, 2014.

\bibitem{Prat}
A.~F. Prat~Waldron.
\newblock {\em Pfaffian line bundles over loop spaces, spin structures and the
  index theorem}.
\newblock ProQuest LLC, Ann Arbor, MI, 2009.
\newblock Thesis (Ph.D.)--University of California, Berkeley.

\bibitem{ClosedSubspaces}
I.~E. Schochetman, R.~L. Smith, and S.-K. Tsui.
\newblock On the closure of the sum of closed subspaces.
\newblock {\em Int. J. Math. Math. Sci.}, 26(5):257--267, 2001.

\bibitem{Segal}
G.~B. Segal.
\newblock The definition of conformal field theory.
\newblock In {\em Differential geometrical methods in theoretical physics
  ({C}omo, 1987)}, volume 250 of {\em NATO Adv. Sci. Inst. Ser. C Math. Phys.
  Sci.}, pages 165--171. Kluwer Acad. Publ., Dordrecht, 1988.

\bibitem{StolzTeichnerElliptic}
S.~Stolz and P.~Teichner.
\newblock What is an elliptic object?
\newblock In {\em Topology, geometry and quantum field theory}, volume 308 of
  {\em London Math. Soc. Lecture Note Ser.}, pages 247--343. Cambridge Univ.
  Press, Cambridge, 2004.

\bibitem{StolzTeichnerField}
S.~Stolz and P.~Teichner.
\newblock Supersymmetric field theories and generalized cohomology.
\newblock In {\em Mathematical foundations of quantum field theory and
  perturbative string theory}, volume~83 of {\em Proc. Sympos. Pure Math.},
  pages 279--340. Amer. Math. Soc., Providence, RI, 2011.

\bibitem{Tener}
J.~E. Tener.
\newblock {\em Construction of the unitary free fermion {S}egal conformal field
  theory}.
\newblock ProQuest LLC, Ann Arbor, MI, 2014.
\newblock Thesis (Ph.D.)--University of California, Berkeley.

\end{thebibliography}
\bibliographystyle{abbrv}

\end{document}